\documentclass[english,a4paper,10pt]{article}
\usepackage[verbose, a4paper]{geometry}
\usepackage{graphicx}%
\usepackage{multirow}%

\usepackage{amsthm,mathtools,amssymb,nicefrac}%
\mathtoolsset{showonlyrefs}
\usepackage{booktabs}  % For prettier tables
\usepackage[unicode=true, bookmarks=true, bookmarksnumbered=false, bookmarksopen=false, breaklinks=false, pdfborder={0 0 0}, pdfborderstyle={}, backref=page, colorlinks=true, pdfauthor={Rajmadan Lakshmanan}, citecolor=ForestGreen, urlcolor= darkgray, linkcolor= ForestGreen]{hyperref}
\usepackage{mathrsfs}%  formal script
\usepackage[round]{natbib}% numbers, rounded 
\usepackage{enumitem}
\usepackage{nowidow}
\setlist[enumerate,1]{label=(\roman*)}
\usepackage{textcomp}%  fonts
\usepackage{nowidow}
\usepackage{dsfont}
\usepackage[ruled, vlined]{algorithm2e}	
\usepackage{todonotes}
\usepackage{csquotes} 
\usepackage{caption}
\usepackage{tikz}
\usepackage{pgfplots}
\usepackage{subcaption}
\pgfplotsset{compat=1.18}
\usetikzlibrary{positioning, calc, arrows.meta}
\usepackage{xcolor}%

\definecolor{ForestGreen}{HTML}{005F50} % 228B22; TU Chemnitz: 005F50
\theoremstyle{plain}
	\newtheorem{theorem}             {Theorem}[section]
	\newtheorem{corollary}  [theorem]{Corollary}
	\newtheorem{proposition}[theorem]{Proposition}
	\newtheorem{lemma}      [theorem]{Lemma}
\theoremstyle{definition}
	\newtheorem{definition} [theorem]{Definition}
	
\theoremstyle{remark}
	\newtheorem{remark}     [theorem]{Remark}
\numberwithin{equation}{section}   % numbering equations

\usepackage{microtype}
\usepackage{doi}
\usepackage{orcidlink}

\DeclareMathOperator{\E}{{\mathds E}}

\DeclareMathOperator{\VaR}{{\mathsf{V@R}}}
\DeclareMathOperator*{\essinf}{ess\,inf}    %   essential supremum
\DeclareMathOperator*{\esssup}{ess\,sup}    %   essential supremum

\newcommand{\one}{{\mathds 1}} 		% \usepackage[sans]{dsfont}

\title{Higher-Order Stochastic Dominance Constraints in Optimization}
\author{%
  Rajmadan Lakshmanan%
		\thanks{Faculty of Mathematics, University of Technology, Chemnitz, Germany}\,
		\footnote{\orcidlink{0009-0006-3273-9063} \href{https://orcid.org/0009-0006-3273-9063}{https://orcid.org/0009-0006-3273-9063}. Contact: \protect\href{rajmadan.lakshmanan@math.tu-chemnitz.de}{rajmadan.lakshmanan@math.tu-chemnitz.de}}
	\and
	Alois Pichler%
		\footnotemark[1]\,
		\thanks{\href{https://orcid.org/0000-0001-8876-2429}{\orcidlink{0000-0001-8876-2429} https://orcid.org/0000-0001-8876-2429}	%		\href{alois.pichler@math.tu-chemnitz.de}{alois.pichler@math.tu-chemnitz.de},
		}
    \and
    Milo\v{s} Kopa%
    \thanks{Charles University, Prague, Czech Republic}\,
}

\begin{document}
	\maketitle
	\begin{abstract}
    This contribution examines optimization problems that involve stochastic dominance constraints. These problems have uncountably many constraints.
    We develop methods to solve the optimization problem by reducing the constraints to a finite set of test points needed to verify stochastic dominance.
    This improves both theoretical understanding and computational efficiency.
    Our approach introduces two formulations of stochastic dominance—one employs expectation operators and another based on risk measures—allowing for efficient verification processes.
    Additionally, we develop an optimization framework incorporating these stochastic dominance constraints.
    Numerical results validate the robustness of our method, showcasing its effectiveness for solving higher-order stochastic dominance problems, with applications to fields such as portfolio optimization.
	\end{abstract}

\section{Introduction}
Stochastic dominance is an important perspective in decision-making under uncertainty and quantitative finance.
It provides a strong method for comparing random variables via their distribution functions. 
The concept facilitates a structured framework for evaluating the superiority of one investment, policy, or strategy over another under conditions of uncertainty.
By employing cumulative distribution functions, stochastic dominance enables decision-makers to manage without presupposing specific utility functions, thus providing a more generalized and mathematically rigorous approach to optimization and risk management.
The precision and depth of stochastic dominance make it an indispensable tool in the analysis of complex probabilistic systems.

Our research introduces a novel method for verifying stochastic dominance, enabling analysis across all orders, including non-integral and higher orders, beyond the second order of stochastic dominance.
The approach reduces the infinitely many, non-linear constraints of the genuine formulation to finitely many, which results in a computationally efficient framework.
This reformulation simplifies the verification process and provides a rigorous mathematical solution to the challenges of establishing stochastic dominance for arbitrary order. 
We extend this framework to a risk functional formulation of stochastic dominance.
%, creating a practical and flexible tool for evaluating distributions.
Numerical examples demonstrate its effectiveness in real-world applications, particularly in portfolio optimization.
Using the MSCI index as a benchmark, we conduct experiments that showcase the reliability and robustness of this method in the portfolio optimization framework.

\subsection{Literature and related works}
\citet{hadar1969rules} pioneered the foundational theories for \emph{First}-Order Stochastic Dominance~(FSD), which serve as a key benchmark in decision-making under uncertainty.
\citet{levy1992stochastic} expanded this field by developing the core principles of Second-Order Stochastic Dominance~(SSD), offering deeper insights into risk and utility analysis.
Building on these seminal works, \citet{RuszOgryczak} build a quantile-based framework for mean-risk models that aligns with second-degree stochastic dominance and develop formulations using stochastic linear programming.
Furthermore, \citet{bellini2007coherent} analyze the consistency between risk measures and higher-order stochastic dominance. 
Their contributions have profoundly enriched both theoretical and practical aspects of decision-making processes.
For a recent textbook overview, we refer to \citet{dentcheva2024risk}.

\citet{dentcheva2003optimization} and \citet{rudolf2008optimization} introduce a robust framework for solving optimization problems with stochastic dominance constraints of \emph{second} order, marking a crucial advancement in the field.
On the computational front, \citet{fabian2011processing} utilize stochastic approximation methods to enhance the computational efficiency of solving SSD~optimization problems.
These contributions have been critical in improving the practicality and performance of optimization methodologies for complex problems.

%Meanwhile, \citet{pflug2007ambiguity} developed decomposition algorithms that substantially reduce the computational burden.
Empirical studies by \citet{roman2013enhanced} demonstrate that incorporating SSD constraints in portfolio optimization leads to portfolios with superior risk-adjusted returns.
For empirical (equiprobable) probabilities, \citet{kuosmanen2004efficient} shows that portfolios constrained by SSD outperform traditional mean-variance optimized portfolios.
This approach builds on the majorization theorem from \citet{hardy1952inequalities} and extends further in \citet{luedtke2008new} to cover a general probability measure setting.
Building on these foundational works, \citet{ThirdKopa} explore how \emph{third}-order stochastic dominance applies to portfolio optimization.
They demonstrate that using this dominance constraint improves portfolio performance.
Recently, \citet{norkin2024portfolio} examine portfolio optimization under first order stochastic dominance and incorporate exact penalty functions to enhance the computational approach. 
%Specifically, it helps to lower the risk of significant losses.

\citet{kopa2023multistage} extend this concept to multistage stochastic dominance, with a particular focus on pension fund management, highlighting the importance of considering the temporal aspects of risk.
\citet{maggioni2016bounds, maggioni2019guaranteed} contribute to SSD by developing robust bounds and approximations for multistage stochastic programs.
They are crucial for ensuring the reliability of solutions in non-discrete, risk-averse optimization contexts.
Moreover, \citet{fabian2008handling} and \citet{dentcheva2012two} address two-stage stochastic optimization problems with ordering constraints on recourse decisions, addressing issues related to feasibility and solution quality in complex decision environments.
Additionally, \citet{consigli2020long,consigli2023asset} examine financial planning and asset-liability management under stochastic dominance, offering practical insights into balancing risk across different stages of investment decision-making.
%Together, these studies underscore the growing recognition of stochastic dominance constraints as a powerful tool in optimizing portfolios with a focus on long-term risk and reward trade-offs.

%These findings underscore the practical benefits and effectiveness of SSD constraints in real-world portfolio management, leading to significant improvements in investment outcomes.

Higher-order stochastic dominance has emerged as a critical tool in optimization and portfolio choice, offering nuanced insights into risk preferences and decision-making under uncertainty.
Research by \cite{dentcheva2013stability} highlights the role of probability metrics in quantifying changes in optimal values and solution sets in higher-order. 
\citet{post2013general} develop linear formulations of higher-order stochastic dominance using piecewise polynomial utility representations.
This approach allows an efficient comparison of prospects and shows that traditional mean-variance methods often fail to identify inefficiencies in market portfolios.
Studies by \citet{fang2017higher} on optimality and efficiency across various degrees of stochastic dominance provide tractable systems of inequalities that capture complex investor preferences.
It reveals the influence of higher-order risk attitudes on investment decisions.
Additionally, investigations by \citet{fang2022optimal} into fourth-degree dominance show that prudent and temperate investors benefit from portfolios emphasizing skewness and minimizing kurtosis.
It results in superior performance through industry-level momentum effects.
These works collectively underscore the value of higher-order dominance in advancing robust, preference-aligned decision frameworks.

% Despite significant advances, traditional methods only address stochastic dominance of order two and current methods often suffer from computational complexity due to the need for handling uncounatably many constraints.
% To address these difficulties, the results of our research allow reducing the constraints to finitely many constraints, and to address stochastic dominance of any order higher, and including two.
% Additionally, we examine the issue through the lens of risk measures and provide numerical expositions.
% We also demonstrate robustness of our approach, reinforcing its efficacy within traditional frameworks such as portfolio optimization.

\paragraph{Outline and Contributions.}
This paper is organized as follows.
Section~\ref{sec:pre} covers the essential definitions and properties of stochastic dominance.
%, establishing the foundational concepts used throughout.
Section~\ref{sec:main} introduces our main theoretical results by analyzing two distinct formulations of stochastic dominance, one based on expectation operators and the other on a risk function.
In this section, we rigorously show how to reduce the test points needed to confirm stochastic dominance from infinitely many to a finite set. 
Our primary results include Theorem~\ref{thm:156} (below), which provides the \emph{finitely} many conditions for verifying stochastic dominance with respect to norms (equivalent to expectation operators), and Theorem~\ref{thm:140} (below), which identifies critical risk levels for verifying stochastic dominance relations using \emph{finitely} many risk levels.
Section~\ref{sec:opti} expands on these findings by introducing an optimization method that includes \emph{non-linear} stochastic dominance constraints and explains a specialized algorithm for its implementation.
Section~\ref{sec:numerics} offers a numerical exposition of our method and demonstrates the effectiveness of Algorithm~\ref{alg:Newton} (below) in addressing higher-order stochastic dominance problems.

\section{Preliminaries\label{sec:pre}}
% This section revisits essential definitions and properties of stochastic dominance.

In the broader sense, this paper examines the optimization problem
\begin{align}
	\text{maximize } & \E Y\label{eq:1} \\
	\text{subject to } & X \preccurlyeq Y\text{ for every }  X \in \mathcal{X},\label{eq:2}
\end{align}
where “$\preccurlyeq$” denotes a partial order between random variables $X$ and~$Y\in\mathcal X$, and $\mathcal{X}$ represents a set of benchmark random variables. More specifically, we consider general stochastic dominance relations of order \( p \geq 1 \) (including \( p  = 1 \) and \( p = 0 \)), denoted by $\preccurlyeq^{(p)}$.

From a portfolio optimization perspective, the feasible random variable~$Y$ in~\eqref{eq:1}  dominates all \emph{benchmark} variables~$X$ and often, the set $\mathcal X$ consists of a single random variable~$X$ so that~\eqref{eq:2} consists in finding a random variable with $X\preccurlyeq Y$, which optimizes the objective~\eqref{eq:1}.
\begin{definition}[Stochastic dominance]\label{def:150}
	Let $\|\cdot\|$ be a norm on a set of random variables.
  It is said that $X$ is stochastically dominated by $Y$, denoted
  \[  X\preccurlyeq^{\|\cdot\|} Y,\] if
  \begin{equation}\label{eq:33}
    \|(X-t)_+\|\ge \|(Y-t)_+\| \text{~ for all }t\in \mathbb R,
  \end{equation}
  where $x_+\coloneqq \max(0,x)$.
\end{definition}
\begin{remark}\label{rem:164}
	The function $t\mapsto \|(t-X)_+\|$ in~\eqref{eq:33} is non-decreasing and convex; indeed, it holds that
  \begin{align}
    (t_\lambda-X)_+&= \bigl((1-\lambda)(t_0-X)+\lambda(t_1-X)\bigr)_+\\
           &\le (1-\lambda)(t_0-X)_+ + \lambda(t_1-X)_+,\label{eq:35}
  \end{align}
  where $t_\lambda\coloneqq (1-\lambda)t_0+\lambda\,t_1$. By the triangle inequality for the norm it follows from~\eqref{eq:35} that
  \begin{align}\label{eq:34}
    \|(X-t_\lambda)_+\|\le (1-\lambda)\|(X-t_0)_+\|+ \lambda\|(X-t_1)_+\|,
  \end{align}
  from which convexity is immediate.
  In what follows, we shall assume that the derivative of $ t\mapsto \|(X-t)_+\|$ exists and is continuous.
\end{remark}
\begin{remark}[Partial order]
	The stochastic dominance relation satisfies
	\begin{enumerate}[noitemsep]
		\item $X\preccurlyeq^{\|\cdot\|} X$ (reflexivity) and
		\item if $X\preccurlyeq^{\|\cdot\|} Y$ and $Y\preccurlyeq^{\|\cdot\|} Z$, then $X\preccurlyeq^{\|\cdot\|} Z$ (transitivity).
	\end{enumerate}
	However, $X\preccurlyeq^{\|\cdot\|} Y$ and $Y\preccurlyeq^{\|\cdot\|} X$ do \emph{not} imply $X=Y$ (i.e., the relation $\preccurlyeq^{\|\cdot\|}$ is not antisymmetric), and consequently, $\preccurlyeq^{\|\cdot\|}$ is not a partial order on the appropriate set of random variables. However, more specifically, the order is defined on the corresponding space of distributions.
\end{remark}
\begin{definition}
  Let $p\in [1,\infty)$ and $X$, $Y\in L^p$ be random variables. It is said that $X$ is dominated by $Y$ in $p$th stochastic order, denoted
  \[  X\preccurlyeq^{(p)} Y, \] if
  \begin{equation}\label{eq:3}
      \E (t-X)_+^{p-1} \ge \E (t-Y)_+^{p-1}\quad\text{for all } t\in \mathbb R,
    \end{equation}
    and
    \begin{equation}\label{eq:40}
      X\preccurlyeq^{(\infty)} Y,\quad \text{ if }\quad  \essinf X\le \essinf Y.
    \end{equation}
  \end{definition}
\begin{remark}[Norm formulation and infinity order]
  \label{rem:NormInfy}
  The condition~\eqref{eq:3} is equivalent to
  \begin{equation}\label{eq:36}
    \|(t-X)_+\|_{p-1} \ge \|(t-Y)_+\|_{p-1},\qquad t\in \mathbb R,
  \end{equation}
  where $\|X\|_p\coloneqq (\E |X|^p)^{\nicefrac1{p}}$ and $\|X\|_\infty\coloneqq \esssup |X|$ (for $p=\infty$).
  For $p= \infty$, the norm in~\eqref{eq:36} is
  \[  \|(t-X)_+\|_\infty= \begin{cases}
   0 &	\text{if } t \le X\text{ a.s.},	\\
  \esssup(t-X)   & \text{else}
  \end{cases}= \begin{cases}
		0 &	\text{if }t\le X\text{ a.s.},	\\
				t-\essinf X   & \text{else},
  \end{cases}\]
  so that the relation $X\preccurlyeq^{(\infty)}Y$ with~\eqref{eq:33} translates into
  \begin{equation}\label{eq:31}
    \essinf X \le \essinf Y,
  \end{equation} which is the defining equation~\eqref{eq:40}.
\end{remark}

  The stochastic dominance of order $p$ and stochastic dominance with respect to the norm $\|\cdot\|_p$ are related by \[  \preccurlyeq^{(p+1)} ~~~ \iff ~~~ \preccurlyeq^{\|\cdot\|_p}.\] The reason for the apparent, unfortunate mismatch between~$p$ and $p+1$ is historic.
\begin{remark}\label{rem:177}
	The stochastic dominance relation $\preccurlyeq^{(\infty)}$ is occasionally defined differently in the literature (cf.\ \citet{SDInftyWhitmore}), but the definition in~\eqref{eq:40} is consistent with Lebesgue norms in~\eqref{eq:36}.
\end{remark}

%\mathrm d
% \begin{lemma}[Cf.\,\citep{Ogryczak99,Ogryczak01}]\label{lem:289}
%   Let $F_X^{(1)}(\cdot) := F_X(\cdot)$. For $k = 2, 3, \dots$, let the $k$th repeated integral be $F_X^{(k)}(x) := \int_{-\infty}^x F_X^{(k-1)}(y) \, \mathrm d y$.  
%   The following two statements are equivalent and they characterize stochastic dominance of integer orders ($k = 1, 2, \dots$) through repeated integrals:
%   \begin{enumerate}
%       \item[(i)] $X \preceq^{(k)} Y$,
%       \item[(ii)] $F_X^{(k)}(x) \geq F_Y^{(k)}(x)$ for all $x \in \mathbb{R}$.
%   \end{enumerate}
% \end{lemma}
% \begin{proof} It holds with Cauchy’s formula for repeated integration that
% \[
% F_X^{(k)}(x) = \frac{1}{(k-2)!} \int_{-\infty}^x (x-y)^{k-2} F_X(y) \, \mathrm d y.
% \]
% By integration by parts, the latter is
% \[
% F_X^{(k)}(x) = \frac{1}{(k-1)!} \int_{-\infty}^x (x-y)^{k-1} \, \mathrm d F_X(y),
% \]
% so that
% \[
% F_X^{(k)}(x) = \frac{1}{(k-1)!} \int_{-\infty}^\infty (x-y)_+^{k-1} \, \mathrm d F_X(y) = \frac{1}{(k-1)!} \mathbb{E} \left( x - X \right)_+^{k-1},
% \]
% from which the assertion follows from the defining condition in Definition~\ref{def:150}.
% \end{proof} 
\begin{remark}[Cf.\,\citep{Ogryczak99,Ogryczak01}]
  Among the many formulations of higher-order stochastic dominance, the ones based on distribution functions are well-known.
  Let \( F_X^{(1)}(\cdot) := F_X(\cdot) \). For \( k = 2, 3, \dots \), the \( k \)-th repeated integral is defined as \( F_X^{(k)}(t) := \int_{-\infty}^t F_X^{(k-1)}(y) \, \mathrm{d}y \).
  The following three statements are equivalent, and they characterize stochastic dominance of integer orders (\( k = 1, 2, \dots \)):
  \begin{enumerate}
      \item \( X \preceq^{(k)} Y \),
      \item \( F_X^{(k)}(t) \geq F_Y^{(k)}(t) \) for all \( t \in \mathbb{R} \),
      \item \( \|(t-X)_+\|_{k-1} \geq \|(t-Y)_+\|_{k-1}\) for all \( t \in \mathbb{R} \).
  \end{enumerate}

  Additionally, we notice that \( X \preceq^{(k)} Y \) implies \( X \preceq^{(k+1)} Y \) for all natural numbers \( k = 1, 2, \dots\,\,\).
  For real numbers \( 1 \leq p \leq p' \), \( X \preceq^{(p)} Y \) implies \( X \preceq^{(p')} Y \).
  As the order increases, both for integer and non-integer orders, the stochastic dominance conditions become less restrictive, offering more flexibility in comparing random variables.
\end{remark}

\section{Characterization of stochastic dominance\label{sec:main}}
This section presents our main results.
We examine two stochastic dominance formulations~-- one using expectation operators and the other based on a risk function.
We rigorously demonstrate a method to reduce the required test points for confirming stochastic dominance from an infinite number to a finite set.
\subsection{Reduction from uncountable to finitely many constraints}
The stochastic order $\preccurlyeq^{\|\cdot\|}$ ($\preccurlyeq^{(p)}$, resp.)\@ is, in practice, a tight requirement. That is, it is often difficult to find random variables~$Y$ such that $X^{\|\cdot\|}$ ($X\preccurlyeq^{(p)}Y$, resp.).
The condition~\eqref{eq:33} requires testing the condition for all $t\in \mathbb R$.
The following theorem reveals that the number of test points may be reduced significantly.
%\todo{p is missing}
\begin{figure}[!htb]
  \centering
  \resizebox{0.42\textwidth}{0.41\textwidth}{
\begin{tikzpicture}[scale=0.23] % Increased the scale for wider figure
  % Grid (optional, can remove if not needed)
  %\draw[very thin, gray] (-8,0) grid[xstep=1.6, ystep=1.6] (8,10.5);
  % Axes
  % \begin{axis}[
  %   width=0.8\textwidth,
  %   height=0.6\textwidth]
  \draw[->] (-15,0) -- (15,0) node[right] {$t$}; % t-axis extended in both directions
  \draw[->] (0,-2) -- (0,40.); % Norm values on y-axis
  
  % Critical points (t1 to t4) on the t-axis
  \foreach \x/\name in { -9/$t_1$, 3/$t_2$, 8/$t_3$ } {
    \node[below] at (\x,0) {\name};
  }

\draw[thick, smooth, blue] plot coordinates {
  (-15, 3.5)
  (-14, 3.7)
  (-13, 4.0)
  (-12, 4.4)
  (-11, 5.0)
  (-9, 6.4)
  (-7, 8.)
  (-4, 10.5)
  (-2, 12.5)
  (0, 14.8)
  (2, 17.5)
  (4, 20.6)
  (6, 24.1)
  (8, 28.5)
  (10, 35.5)
  (10.5, 37.3)
  (11, 39.)  
};
\draw[thick, smooth, red] plot coordinates {
  (-15, 3.1)
  (-14, 3.4)
  (-13, 3.7)
  (-12, 4.2)
  (-11, 4.9)
  (-9, 6.4)    % Touching point
  (-7, 7.9)
  (-4, 10.3)
  (-2, 12.0)
  (0, 13.8)
  (2, 16.3)
  (4, 19.3)
  (6, 22.9)
  (8, 28.1)    % Touching point
  (10, 33.5)
  (10.5, 34.8)
  (11.2, 36.5)  
};
\draw[thin, smooth, red, dashed] plot coordinates {
  (-15, 2.5)     % Adjusted downward
  (-14, 2.7)
  (-13, 3.0)
  (-12, 3.4)
  (-11, 3.9)
  (-9, 5.0)      % Linear starts here
  (-7, 6.2)
  (-4, 8.1)
  (-2, 9.6)     % Transition to quadratic
  (0, 11.)
  (2, 13.1)
  (3, 14.8)
  (4, 16.5)
  (6, 19.8)
  (8, 23.4)
  (10, 27.1)
  (12, 30.6)   
};

\draw[thin, smooth, blue, dashed] plot coordinates {
  (-15, 2.2)   % Start point
  (-14, 2.5)
  (-13, 2.8)
  (-12, 3.2)
  (-11, 3.7)
  (-9, 5.0)      % Transition point
  (-7, 6.3)
  (-4, 8.5)
  (-2, 9.9)
  (0, 11.8)
  (2, 13.8)
  (3, 14.8)
  %(3.5, 15.5)
  (4, 15.9)
  (6, 19.2)
  (7, 21.3)
  (8, 23.4)
  (10, 28.8)
  (12, 33.9)  
  %(10, 27.9)
  % (12, 31.9)  
};
  % Shaded regions with intersections at boundaries
   \fill[red!20, opacity=0.3] (-15.,0) rectangle (-9,39.); % Interval [t1, t2]
   \fill[blue!20, opacity=0.3] (-9,0) rectangle (2.7,39.); % Interval [t2, t3]
  %\fill[blue!20, opacity=0.7] (-1,0) rectangle (2,17.); % Interval [t3, t4]
   \fill[red!20, opacity=0.3] (2.7,0) rectangle (8,39.); % Interval [t4, t5]
   \fill[blue!20, opacity=0.3] (8,0) rectangle (12,39.); % Interval [t5, t6]
\end{tikzpicture}}
\begin{tikzpicture}[scale=0.1]
% Legend
% Horizontal legend under the main plot
\node[align=center] at (0,-3.5) {
  \begin{tabular}{cccccccc}
    \begin{tikzpicture}
      \draw[blue, thick] (0,0) -- (0.5,0);
    \end{tikzpicture} & $\|(t - X)_+\|$ &
    \begin{tikzpicture}
      \draw[red, thick] (0,0) -- (0.5,0);
    \end{tikzpicture} & $\|(t - Y)_+\|$ %\\[1mm]
    \begin{tikzpicture}
      \draw[blue, thin, dashed] (0,0) -- (0.5,0);
    \end{tikzpicture} & $\frac{d}{dt} \|(t - X)_+\|$ &
    \begin{tikzpicture}
      \draw[red, thin, dashed] (0,0) -- (0.5,0);
    \end{tikzpicture} & $\frac{d}{dt} \|(t - Y)_+\|$
  \end{tabular}
};
\end{tikzpicture}
\caption{Illustration of the stochastic dominance verification condition from Theorem \ref{thm:156}. 
The solid curves \( \|(t - X)_+\| \) (blue) and \( \|(t - Y)_+\| \) (red) represent the norms under comparison, while the dashed curves show their respective derivatives, \( \frac{d}{dt} \|(t - X)_+\| \) (blue, dashed) and \( \frac{d}{dt} \|(t - Y)_+\| \) (red, dashed).
The critical points \( t_1 \), \( t_2 \), and \( t_3 \) are identified at intersections of the derivatives, and they provide sufficient confirmation of the dominance relationship $ X \preccurlyeq^{\|\cdot\|} Y$.
}  
\label{fig:stochastic_dominance_extended}
\end{figure}
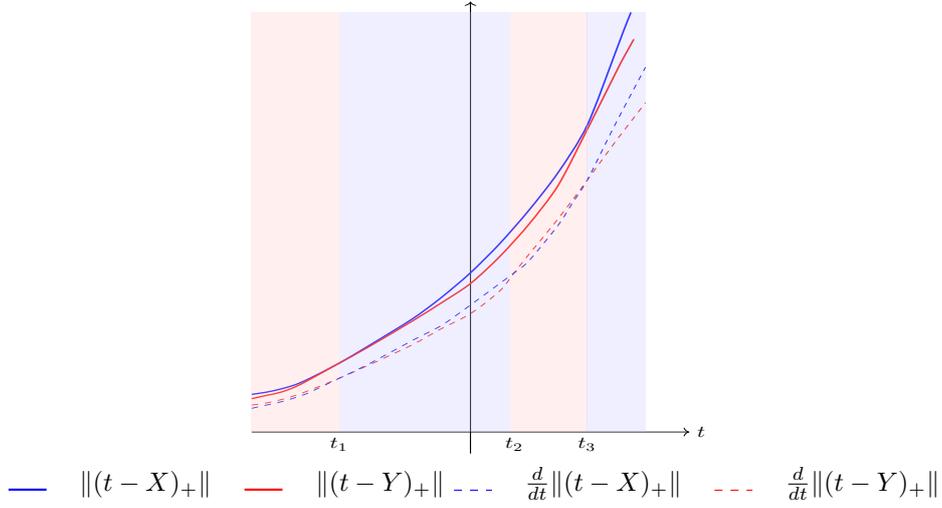
\begin{theorem}\label{thm:156}
	To verify stochastic dominance with respect to the norm $\|\cdot\|$ it is sufficient to verify
	\begin{equation}\label{eq:38}  \|(t-X)_+\|\ge \|(t-Y)_+\| \end{equation}
	for those distincitve $t\in\mathbb R$, for which
  \begin{equation}\label{eq:37}   \partial_t \|(t-X)_+\|~ \cap~ \partial_t\|(t-Y)_+\|\end{equation}
  is not empty.
\end{theorem}
\begin{remark}\label{rem:480}
    The functions in the preceding theorem are convex, so they have a non-empty subgradient by Remark~\ref{rem:164}. This subgradient with respect to the variable~$t$ is denoted~$\partial_t$.
\end{remark}

\begin{proof}
  To abbreviate the notation define the functions $g_X(t)\coloneqq \|(t-X)_+\|$ and $g_Y(t)\coloneqq \|(t-Y)_+\|$ with derivatives (multifunctions)
  \begin{equation}\label{eq:44}  g^\prime_X(t)\in \partial_t\,g_X(t) \text{ and } g^\prime_Y(t)\in \partial_t\,g_Y(t).\end{equation}
  By Remark~\ref{rem:164}, the selections $g^\prime_X$ and $g^\prime_Y$ are nonnegative, that is,
  \[  g_X^\prime(t), \ g_Y^\prime(t)\ge 0,\qquad t\in\mathbb R, \] and we assume for the moment that both functions are continuous.

  \begin{enumerate}
    \item
  For $t$ given, assume first that $g^\prime_X(t)> g^\prime_Y(t)$. Define
  \[  s\coloneqq \sup\bigl\{s^\prime<t\colon g^\prime_X(s^\prime)\le g^\prime_Y(s^\prime)\bigr\}\] and observe that $s<t$, as the functions $g^\prime_X$ and $g^\prime_Y$ are continuous, and $g^\prime_X(u)\ge g^\prime_Y(u)$ for $u\in (s,t)$; on the boundary of this interval, it holds, by continuity, that $g^\prime_X(s)= g^\prime_Y(s)$.  We conclude from~\eqref{eq:38} and~\eqref{eq:37} $g_X(s)\ge g_Y(s)$ at~$s$, and thus
  \begin{align}\label{eq:39}
    g_X(t)&= g_X(s)+ \int_s^t g_X^\prime(u)\,\mathrm d u\\
            &\ge  g_Y(s)+ \int_s^t g_Y^\prime(u)\,\mathrm d u\\
            & = g_Y(t)
  \end{align}
  and thus~\eqref{eq:33} at~$t$.

  \item Assume next that $g^\prime_X(t)< g^\prime_Y(t)$. Define
  \[  s\coloneqq \inf\bigl\{s^\prime>t\colon g^\prime_X(s^\prime)\ge g^\prime_Y(s^\prime)\bigr\}\] and observe that $t<s$. The infimum is attained, as the functions $g^\prime_X$ and $g^\prime_Y$ are continuous, and it holds that $g^\prime_X(s)= g^\prime_Y(s)$. We conclude from~\eqref{eq:38} and~\eqref{eq:37} that $g_X(s)\ge g_Y(s)$, and together with the definition of~$s$ that
  \begin{align}
    g_X(t)&= g_X(s)- \int_t^s g_X^\prime(u)\,\mathrm d u\\
            &\ge  g_Y(s)- \int_t^s g_Y^\prime(u)\,\mathrm d u\\
            & = g_Y(t)
  \end{align}
  and thus~\eqref{eq:33} at~$t$.
  \end{enumerate}
  This completes the proof.
\end{proof}

\begin{remark}\label{rem:167}
  The assertion of Theorem~\ref{thm:156} does not apply for $p=1$, as in this case the functions $g^\prime_X$ and $g^\prime_Y$ in~\eqref{eq:44} are \emph{not} continuous, so that there are possible no solutions of the characterizing equation~\eqref{eq:37}.
\end{remark}

\subsection{Characterization with risk measures}
This section provides an additional characterization of the general stochastic dominance relation by involving risk functionals, which is in some sense dual.
To this end we revisit essential definitions of risk measures and their properties, while also presenting optimization problems in terms of risk functionals along with the reduced constraints~\eqref{eq:38}.

%We shall employ the following characterization or the stochastic dominance relation~\eqref{eq:3}.
% \begin{definition}[Risk measure]\label{def:141}
%   The higher order risk measure is
%   \begin{equation}
%     \mathcal R_\beta(Y)\coloneqq \inf_{t\in\mathbb R}\ t+{1 \over 1-\beta} \|(Y-t)_+\|_p.
%   \end{equation}
% \end{definition}
\begin{definition}[Risk measure]\label{def:141}
  Let \(\beta \in (0, 1)\) be a risk parameter, \(Y\) a random variable representing potential losses. The \emph{higher order risk measure} at risk level~$\beta$ is
  \begin{equation}\label{eq:4}
    \mathcal{R}_\beta(Y) \coloneqq \inf_{t \in \mathbb{R}} \  t + \frac{1}{1 - \beta} \|(Y - t)_+\|.
  \end{equation}
\end{definition}
The higher order risk measure $\mathcal R_\beta$ is a risk measure, satisfying the following axioms of general risk measures.
\begin{definition}[Cf.\ \citet{Artzner1997, Artzner1999}]\label{def:106}
	The function $\mathcal R\colon \mathcal X\to \mathbb R$ is a risk measure, if
  \begin{enumerate}[nolistsep, noitemsep]
    \item $\mathcal R(\lambda\,X)= \lambda\,\mathcal R(X)$ for all $\lambda\ge 0$,
    \item $\mathcal R(X+c)= c+\mathcal R(X)$ for all $c\in\mathbb R$,
    \item $\mathcal R(X+Y)\le \mathcal R(X)+ \mathcal R(Y)$ for all $X$, $Y\in\mathcal X$, and
    \item\label{enu:1} $\mathcal R(X)\le \mathcal R(Y)$, provided that $X\le Y$ almost everywhere.
  \end{enumerate}
\end{definition}
The relation “$X\le Y$ almost surely” in~\ref{enu:1} is occasionally also stated as “$X\preccurlyeq^{(0)}Y$” and referred to as \emph{stochastic dominance of order~$0$}.

It is demonstrated in \citet{Pichler2024} that the stochastic dominance constraint~\eqref{eq:33} can be expressed equivalently by employing the risk functional $\mathcal R(\cdot)$.
Indeed, the following holds true.
\begin{theorem}\label{thm:3}
  The following are equivalent:
  \begin{enumerate}[nolistsep]
      \item $X\preccurlyeq^{\|\cdot\|}Y$,
      \item \label{enu:5}$\mathcal R_\beta(-X)\ge \mathcal R_\beta(-Y)$ for all $\beta\in (0,1)$.
  \end{enumerate}
\end{theorem}
The proof of Theorem~\ref{thm:3} rigorously characterizes stochastic dominance using the framework of higher-order risk measures and their norm-based properties.
It establishes that $X \preccurlyeq^{\|\cdot\|} Y$ is equivalent to the inequality
\[	 \mathcal R_\beta(-X) \geq \mathcal R_\beta(-Y)  \text{ for all }  \beta \in (0, 1). \]
The equivalence is derived through norm-based conditions \( \| (t - X)_+ \| \geq \| (t - Y)_+ \| \) for all \( t \), linked to risk functional properties in the associated Banach space.
Dual space representations further strengthen the characterization by incorporating constraints from the positive cone of the dual norm.
This integration of risk measures, norms, and duality forms the technical backbone of the theorem.

 % ╭─────────────────────────────────────────────
\subsection{Stochastic dominance in numerical computations}\label{sec:662}
To verify that $X\preccurlyeq^{\|\cdot\|} Y$ it is necessary to verify the defining condition~\eqref{eq:3} for every $t\in \mathbb R$, or~\ref{enu:5} in Theorem~\ref{thm:3} above for every $\beta\in(0,1)$.
In both cases, these are infinitely many comparisons, intractable for numerical computations.
% The same holds true for the equivalent characterization~\ref{enu:5} in Theorem~\ref{thm:3}, as all risk levels $\beta\in (0,1)$ -- again infinitely many -- need to be considered.
This is difficult, perhaps impossible to ensure in numerical computations.

In what follows we develop an equivalent characterization, which builds on only \emph{finitely many} risk levels. With this,  the comparison $X\preccurlyeq^{\|\cdot\|} Y$ is numerically tractable.

We start the exposition with the following lemma on convexity (concavity).
\begin{lemma}\label{lem:150}
  For~$Y$ fixed in the domain of $\mathcal R$, the mapping \[ \beta\mapsto (1-\beta)\cdot\mathcal R_\beta(Y)\] is concave.
\end{lemma}
\begin{proof}
For $\lambda\in(0,1)$, define $\beta_\lambda\coloneqq (1-\lambda)\beta_0+\lambda\,\beta_1$. Choose $t_\lambda$ in~\eqref{eq:4} minimizing $\mathcal R_{\beta_\lambda}(Y)$. Then
\begin{align}
  (1-\beta_\lambda)\,\mathcal R_{\beta_\lambda}(Y)  &= (1-\beta_\lambda) t_\lambda + \|(Y-t_\lambda)_+\| \\
                              &= (1-\lambda)\bigl((1-\beta_0)t_\lambda+ \|(Y-t_\lambda)_+\|\bigr) +\lambda\bigl((1-\beta_1)t_\lambda+\|(Y-t_\lambda)_+\|\bigr) \\
                              & \ge(1-\lambda)(1-\beta_0)\,\mathcal R_{\beta_0}(Y)+ \lambda(1-\beta_1)\,\mathcal R_{\beta_1}(Y)
\end{align}
by passing to the infimum in $\mathcal R_{\beta_0}$ and $\mathcal R_{\beta_1}$ in the latter expression.
\end{proof}

The latter result on convexity leads to the following result on the derivative of the risk functional with respect to the risk level.
\begin{theorem}\label{thm:673}
For $Y$ in the domain of $\mathcal R$ and $\beta\in(0,1)$, the derivative with respect to the risk rate is
\[ {\mathrm d \over \mathrm d\beta} \mathcal R_\beta(Y)= \frac{\mathcal R_\beta(Y)-t_Y(\beta)}{1-\beta}, \]
where $t_Y(\beta)$ minimizes the higher order risk measure $\mathcal R_\beta(Y)$, cf.~\eqref{eq:4}.
\end{theorem}
\begin{proof}
  As in Lemma~\ref{lem:150} above consider the objective
  \begin{equation}
  f(\beta,t)\coloneqq (1-\beta)t+\|(Y-t)_+\|.
  \end{equation}
    For $\beta$ fixed, the objective is concave, and we may choose $t(\beta)$ optimal in~\eqref{eq:3}.
    Define the function
    \begin{equation}\label{eq:46}
        f(\beta)\coloneqq f\bigl(\beta,t(\beta)\bigr)= \min_{t\in\mathbb R}\  f(\beta,t) = (1-\beta)\cdot\mathcal R_\beta(Y).
    \end{equation}
    It holds that
    \begin{align}
        f^\prime(\beta)&= {\partial \over \partial\beta} f\bigl(\beta,t(\beta)\bigr)+ {\partial \over \partial t} f\bigl(\beta,t(\beta)\bigr)\cdot t^\prime(\beta).\\
               & = -t(\beta)+ {\partial \over \partial t} f\bigl(\beta,t(\beta)\bigr)\cdot t^\prime(\beta)\\
                 &= -t(\beta),    \label{eq:147}
    \end{align}
    as $t(\beta)$ is optimal in~\eqref{eq:3} and consequently $0\in \partial_t f\bigl(\beta,t(\beta)\bigr)$.
    For the risk functional it follows that
    \[{\mathrm d \over \mathrm d\beta}\mathcal R_\beta(Y)={\mathrm d\over \mathrm d\beta}{f(\beta) \over 1-\beta}= {-t(\beta)(1-\beta)+ f(\beta) \over (1-\beta)^2}= {\mathcal R_\beta(Y)-t(\beta) \over 1-\beta},\]
    the assertion.
\end{proof}
\begin{corollary}
Let $t_Y(\beta)$ be the minimizer for $\mathcal R_\beta(Y)$ in~\eqref{eq:3}. The function \[\beta\mapsto t_Y(\beta)\] is non-decreasing.
\end{corollary}
\begin{proof}
The function~$f$ introduced in~\eqref{eq:46} is concave by Lemma~\ref{lem:150}, that is, $f^\prime(\cdot)$ is non-increasing. It follows with~\eqref{eq:147} that $t_Y(\cdot)$ is non-decreasing, the assertion.
\end{proof}
\begin{theorem}[Verification of stochastic dominance relations]\label{thm:140}
To verify that $X\preccurlyeq^{\|\cdot\|} Y$ it is sufficient to verify
\begin{equation}\label{eq:24}   \mathcal R_{\beta_i}(-X)\ge \mathcal R_{\beta_i}(-Y)
  \end{equation}
  for the (\emph{finitely many}) risk levels
  \[  \beta_i,\quad i= 1, \dots, n;\]
the critical risk levels $\beta_i$  are chosen so that $\beta_i < \gamma_i < \beta_{i+1}$, where
\begin{align}
  t_{-X}(\beta)& \le t_{-Y}(\beta) \text{ for } \beta\in (\beta_i,\gamma_i) \text{ and}\label{eq:25}\\
   t_{-X}(\beta)& \ge t_{-Y}(\beta) \text{ for } \beta\in (\gamma_i,\beta_{i+1})\label{eq:26}
\end{align}
for $i=1, 2,\dots, n$.
\end{theorem}
% \begin{tikzpicture}
%   % Draw the line for beta
%   \draw[thick] (0,0) -- (10,0);

%   % Beta_i points
%   \node[below] at (2,0) {$\beta_i$};
%   \draw[fill] (2,0) circle [radius=0.05];

%   \node[below] at (5,0) {$\gamma_i$};
%   \draw[fill] (5,0) circle [radius=0.05];

%   \node[below] at (8,0) {$\beta_{i+1}$};
%   \draw[fill] (8,0) circle [radius=0.05];

%   % Draw intervals
%   \draw[<->] (2, 0.5) -- (5, 0.5) node[midway, above] {$\beta \in (\beta_i, \gamma_i)$};
%   \draw[<->] (5, 0.5) -- (8, 0.5) node[midway, above] {$\beta \in (\gamma_i, \beta_{i+1})$};
% \end{tikzpicture}
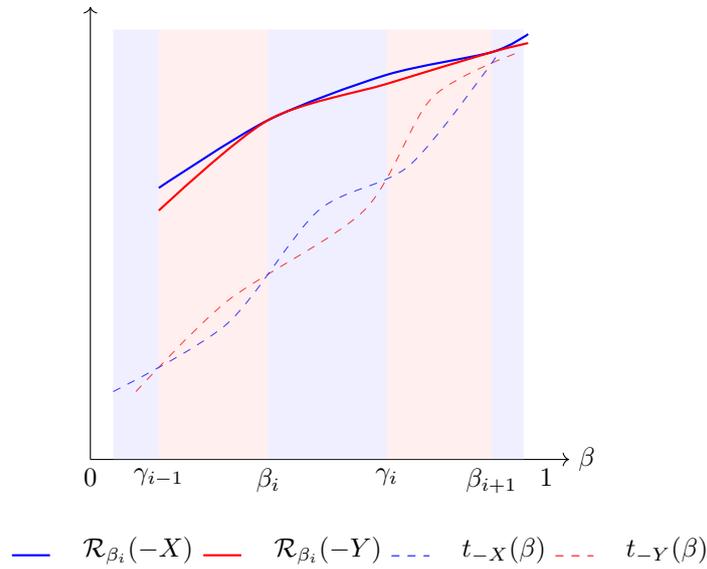
\begin{figure}[!htb]
  \centering
\begin{tikzpicture}[scale=0.60]
  \centering
        % Grid
        %\draw[very thin, gray] (0.5,0) grid[xstep=1.2, ystep=1.2] (10,9.5);
      
        % Axes
        \draw[->] (0.0,0) -- (10.5,0) node[right] {$\beta$};
        \draw[->] (0.0,0) -- (0.0,10.0) node[above] {};
      
        % % Risk levels and critical points
         \foreach \x/\name in {0/$0$,1.5/$\gamma_{i-1}$, 3.9/$\beta_i$, 6.5/$\gamma_i$, 8.8/$\beta_{i+1}$,10/$1$} {
                  \node[below] at (\x,0) {\name};
         }
      
        % Curves representing t_x(β) and t_y(β)
        \draw[thin, smooth, red, dashed] plot coordinates {(1,1.5)(3,3.5) (6,5.5) (7.5,8.) (9.4,9)};
        \draw[thin, smooth, blue, dashed] plot coordinates { (0.5,1.5) (3,3) (5,5.5) (7,6.5) (8.9,8.9)};
      
        % Shading between beta_i-1 and gamma_i-1 (blue)
        \fill[blue!20, opacity=0.3] (0,0) -- plot[smooth, domain=0:3] (0.5,0) -- (0.5,9.5) -- (1.5,9.5) -- (1.5,0) -- cycle;
        
        % Shading between gamma_{i-1} and beta_i (blue)
        \fill[red!20, opacity=0.3] (-1,0) -- plot[smooth, domain=-1:0] (1.5,0) -- (1.5,9.5) -- (3.9,9.5) -- (3.9,0) -- cycle;
      
        % Shading between beta_i and gamma_i (blue)
        \fill[blue!20, opacity=0.3] (0,0) -- plot[smooth, domain=0:3] (3.9,0) -- (3.9,9.5) -- (6.5,9.5) -- (6.5,0) -- cycle;
      
        % Shading between gamma_i and beta_{i+1} (red)
        \fill[red!20, opacity=0.3] (3,0) -- plot[smooth, domain=3:6] (6.5,0) -- (6.5,9.5) -- (8.8,9.5) -- (8.8,0) -- cycle;
      
        % Shading between beta_i+1 and gamma_i+1 (blue)
        \fill[blue!20, opacity=0.3] (0,0) -- plot[smooth, domain=0:3] (8.8,0) -- (8.8,9.5) -- (9.5,9.5) -- (9.5,0) -- cycle;
        % Adding the two monotonic increasing functions R_beta_i(-X) and R_beta_i(-Y)
        \draw[thick, smooth, blue] plot coordinates {(1.5,6)(3.9,7.5) (6.5,8.5) (8.8,9)  (9.6,9.4)};
        \draw[thick, smooth, red ] plot coordinates {(1.5,5.5)(3.9,7.5) (6.5,8.3) (8.8,9) (9.6,9.2)};
      
        % % Mathematical steps boxes
        % \node[draw, fill=white, align=center, minimum width=3.5cm] (step1) at (13.5,6) {Find $\gamma_i$ \\ such that \\ $\beta_i < \gamma_i < \beta_{i+1}$};
      
        % Legend Box
       \node[align=center] at (6,-2) {
          \begin{tabular}{cccccccc}
            \begin{tikzpicture}
                \draw[blue, thick] (0,0) -- (0.5,0);
            \end{tikzpicture} & $\mathcal{R}_{\beta_i}(-X)$ 
            \begin{tikzpicture}
                \draw[red, thick] (0,0) -- (0.5,0);
            \end{tikzpicture} & $\mathcal{R}_{\beta_i}(-Y)$ 
            \begin{tikzpicture}
              \draw[blue, thin, dashed] (0,0) -- (0.5,0);
          \end{tikzpicture} & $t_{-X}(\beta)$ 
          \begin{tikzpicture}
              \draw[red, thin, dashed] (0,0) -- (0.5,0);
          \end{tikzpicture} & $t_{-Y}(\beta)$ 
          \end{tabular}
        };
        % Arrows to steps
      %\draw[->, thick] (6.5,6.2) -- (step1.west);
    \end{tikzpicture}
    \caption{Verification of stochastic dominance relations based on Theorem \ref{thm:140}.
    The plot illustrates \( t_{-X}(\beta) \) and \( t_{-Y}(\beta) \) (blue dashed and red dashed), respectively, for the risk levels.
    These functions provide a visual representation of the conditions in Equations \eqref{eq:25} and \eqref{eq:26}.
    The shaded regions represent dominance intervals, while the solid curves \( \mathcal{R}_{\beta_i}(-X) \) (blue) and \( \mathcal{R}_{\beta_i}(-Y) \) (red) show the relative dominance of \( X \) and \( Y \) at the critical risk levels \( \beta_i \).
    The critical points, denoted by \( \beta_i \) and \( \gamma_i \), satisfy the dominance conditions required to verify \( X \preccurlyeq^{\|\cdot\|} Y \).}

  \label{fig:RiskThm}
\end{figure}
\begin{remark}\label{rem:217}
For the crucial risk levels (test points) $\beta_i$ in the preceding Theorem~\ref{thm:140} it holds that
\[  t_{-X}(\beta)\ge t_{-Y}(\beta) \text{ for }\beta\in (\gamma_{i-1}, \beta_i) \text{ and } t_{-X}(\beta)\le t_{-Y}(\beta) \text{ for }\beta\in (\beta_i, \gamma_i),\]
so that the curves $t_{-X}(\cdot)$ and $t_{-Y}(\cdot)$ intersect at~$\beta_i$, where $t_{-X}(\cdot)$ passes $t_{-Y}(\cdot)$ \emph{from above} at~$\beta_i$, while $t_{-X}$ passes $t_{-Y}(\cdot)$ \emph{from below} at~$\gamma_i$.
For continuous $t_{-X}(\cdot)$ and $t_{-Y}(\cdot)$, the crucial points in the preceding Theorem~\ref{thm:140} are among the points, where the minimizers of $\mathcal R_\beta(-X)$ and $\mathcal R_\beta(-Y)$ coincide, that is,
\[ t_{-X}(\beta_i) = t_{-Y}(\beta_i), \quad i=1,\dots,n;\]
with a similar argument it holds as well that
\[ t_{-X}(\gamma_i) = t_{-Y}(\gamma_i), \quad i=1,\dots,n.\] Refer to Figure~\ref{fig:RiskThm} for a visual representation of the core idea.
\end{remark}

\begin{proof}[Proof of Theorem~\ref{thm:140}]
Recall first from Theorem~\ref{thm:3}\,\ref{enu:5} that $X\preccurlyeq^{(p)} Y$ is equivalent to
\[\mathcal R_\beta(-X)\ge\mathcal R_\beta(-Y)\text{ for all }\beta\in(0,1).\]

Assuming~\eqref{eq:24}, we demonstrate that $\mathcal R_\beta(-X) \ge \mathcal R_\beta(-Y)$ for $\beta\in (\beta_i,\beta_{i+1})$ for $i=1,\dots,n$.
To this end we distinguish the following two cases, where we reuse the notation introduced in the proof of Theorem~\ref{thm:673}. As in the proof of  Theorem~\ref{thm:156} above, we distinguish two cases.
\begin{enumerate}
  \item Suppose that $\beta\in (\beta_i,\gamma_i)$, then, by~\eqref{eq:25}, $t_{-X}(\beta)\le t_{-Y}(\beta)$.
        With~\eqref{eq:47} we conclude that
        \begin{align}
        f_{-X}(\beta)&= f_{-X}(\beta_i)+\int_{\beta_i}^\beta f_{-X}^\prime(\gamma)\,\mathrm d\gamma\\
                  &=  f_{-X}(\beta_i)-\int_{\beta_i}^\beta t_{-X}(\gamma)\,\mathrm d\gamma\\
                  &\ge  f_{-Y}(\beta_i)-\int_{\beta_i}^\beta t_{-Y}(\gamma)\,\mathrm d\gamma  \label{eq:27}\\
                  &= f_{-Y}(\beta_i)+\int_{\beta_i}^\beta f_{-Y}^\prime(\gamma)\,\mathrm d\gamma\\
                  &= f_{-Y}(\beta),
        \end{align}
        where we have used~\eqref{eq:24} in~\eqref{eq:27}. It follows that $\mathcal R_\beta(-X)\ge \mathcal R_\beta(-Y)$ for all $\beta\in(\beta_i,\gamma_i)$.
  \item If $\beta\in (\gamma_i,\beta_{i+1})$, then $t_{-X}(\beta)\ge t_{-Y}(\beta)$ by assumption~\eqref{eq:26}. It holds that
      \begin{align}
        f_{-X}(\beta) &= f_{-X}(\beta_{i+1})- \int_\beta^{\beta_{i+1}}f^\prime_{-X}(\gamma)\mathrm d\gamma\\
        &= f_{-X}(\beta_{i+1})+ \int_\beta^{\beta_{i+1}}t_{-X}(\gamma)\mathrm d\gamma \\
        &\ge   f_{-Y}(\beta_{i+1})+ \int_\beta^{\beta_{i+1}}t_{-Y}(\gamma)\mathrm d\gamma \label{eq:28}\\
        &= f_{-Y}(\beta_{i+1})- \int_\beta^{\beta_{i+1}}f^\prime_{-Y}(\gamma) \mathrm d\gamma \\
        &= f_{-Y}(\beta),
      \end{align}
      where again~\eqref{eq:24} was used in~\eqref{eq:28}. It follows that $\mathcal R_\beta(-X)\ge \mathcal R_\beta(-Y)$ for $\beta\in (\gamma_i,\beta_{i+1})$, the remaining case.
\end{enumerate}
Combining the two cases above we find that $\mathcal R_\beta(-X)\ge \mathcal R_\beta(-Y)$ for all $\beta\in (\beta_i,\beta_{i+1})$.
The assertion for all $\beta\in \mathbb R$ thus follows by considering $i=1,\dots,n$.
\end{proof}
\begin{remark}\label{rem:759}
It is importan to note that the critical risk levels $\beta_i$, $i=1,\dots,n$, in~\eqref{eq:24}~-- in general~-- depend on \emph{both} random variables, on $X$ \emph{and} $Y$.
The following remark presents a notable exception to this rule.
\end{remark}
\begin{remark}[Average value-at-risk]\label{rem:225}
For the Lebesgue norm $\|\cdot\|_p$ with $p=1$, the optimizers are $t_{-X}(\beta)= \VaR_\beta(-X)$, and the function $t_{-X}(\cdot)$ is not continuous for discrete distributions. As $\beta\mapsto t_{-X}(\beta)$ is non-decreasing, the test points are
\[  \beta_i\in \bigl\{F_{-X}(\beta)\colon \beta\in \mathbb R\bigr\},\quad i=1,2,\dots,n.\]
In this case, the crucial risk levels $\beta_i$, $i=1,\dots,n$, are independent of the variable~$Y$ when comparing $X\preccurlyeq^{(p)} Y$. This is observed the first time in \cite{dentcheva2003optimization}.
\end{remark}

\section{Optimization with stochastic dominance constraints\label{sec:opti}}
This section introduces an optimization technique for \emph{non-linear} stochastic dominance constraints, building on the results discussed in the previous section.

\subsection{Problem setup}
Let
\[  \mathcal S\coloneqq \left\{(x_1,\dots,x_n)\colon x_i\ge 0 \text{ for } i=1,\dots,n,\text{ and } \sum_{i=1}^n x_i= 1\right\}\]
be the simplex.
The optimization problem~\eqref{eq:1}--\eqref{eq:2} can be restated as
\begin{align}\label{eq:45}
	\text{minimize } & \mathcal R(-x^\top\xi) \\
	\text{subject to } & \xi_0 \preccurlyeq^{(p)} x^\top\xi, \label{eq:49}\\
                              & x\in \mathcal S, \label{eq:32}
\end{align}
where $\xi_0$ is the return of a benchmark portfolio.
The problem itself is known in portfolio optimization.
Note, that the problem~\eqref{eq:45} is convex, as the objective is convex for $\mathcal R$ being a convex risk measure, and the constraints are convex by Proposition~\ref{prop:167} (see below).

By the definition in the preceding section, the optimization problem can be recast as
\begin{align}
	\text{minimize } & \mathcal R(-x^\top\xi) \\
	\text{subject to } & x\in \mathcal S, \text{ and}\\
                                &  g_p(t,x^\top\xi) \le 0 \text{ for all }t\in \mathbb R\label{eq:542},
\end{align}
where \begin{equation}\label{eq:531} g_p(t,X)\coloneqq \E(t-X)_+^p- \E(t-\xi_0)_+^p \end{equation} and~$X$ is a random variable and~$\xi_0$ the benchmark.

The following proposition demonstrates the convexity of constraint~\eqref{eq:49}. The result is stated slightly more general, involving an additional function~$h$.
\begin{proposition}\label{prop:167}
  Let $h$ be uniformly concave, i.e.,
  \[  h\bigl((1-\lambda)x_0+\lambda\,x_1,\xi\bigr)\ge (1-\lambda)h(x_0,\xi)+\lambda\,h(x_1,\xi) \]
  for all $\xi\in \Xi$ and $\lambda\in(0,1)$.
  For a random variable $\xi_0$ and $p\ge 1$, the set
  \[  \mathcal C\coloneqq \bigl\{x\colon \xi_0 \preccurlyeq^{\|\cdot\|} h(x;\xi)\bigr\}\]
  is convex.
\end{proposition}
\begin{proof}
  Define $x_\lambda\coloneqq (1-\lambda)x_0+\lambda\,x_1$.
  By the assumption on concavity of $h$, \[  h(x_\lambda,\xi)\ge (1-\lambda) h(x_0,\xi)+ \lambda\,h(x_1,\xi)\] and thus, by monotonicity and convexity of $x_+$,
  \begin{align}
    \bigl(t-h(x_\lambda,\xi)\bigr)_+&\le \bigl((1-\lambda)t+ \lambda\,t- (1-\lambda)h(x_0,\xi)- \lambda\,h(x_1,\xi)\bigr)_+\\
                        &\le (1-\lambda)\bigl(t-h(x_0,\xi)\bigr)_++ \lambda\bigl(t-h(x_1,\xi)\bigr)_+.\label{eq:41}
  \end{align}
  For $x_0$, $x_1\in \mathcal C$ it holds that
  \begin{equation}\label{eq:42}
    \|(t-\xi_0)_+\|\ge \|\bigl(t-h(x_i,\xi)\bigr)_+\|, \qquad i=1,\,2,
  \end{equation}
  and thus combining with~\eqref{eq:41}
  \begin{align}\label{eq:43}
    \| \bigl(t-\xi_0\bigr)_+\|&\ge (1-\lambda)\|\bigl(t-h(x_0,\xi)\bigr)_+\|+ \lambda \|\bigl(t-h(x_1,\xi)\bigr)_+\|\\
                            &\ge \|\bigl(t-h(x_\lambda,\xi)\bigr)_+\|.
  \end{align}
  Hence, the assertion.
\end{proof}
Using Theorem~\ref{thm:156}, we may express problem~\eqref{eq:45} as
\begin{align}\label{eq:146}
	\text{minimize } & \mathcal{R}(-x^\top \xi) \\
	\text{subject to } & x \in \mathcal{S}, \\
                                & g_p(t, x^\top \xi) \le 0,\label{eq:834} \\
                              & \qquad \text{for all } t \text{ such that } g_{p-1}(t, x^\top \xi) = 0. \label{eq:578} 
\end{align}
This optimization problem~\eqref{eq:146} is \emph{non-linear} and includes \emph{implicit} inequality constraints. The objective function is convex, and, by Proposition~\ref{prop:167}, the constraint~\eqref{eq:49} is also convex.
% With Theorem~\ref{thm:156}, the problem~\eqref{eq:45} is
% \begin{align}\label{eq:146}
% 	\text{minimize } & \mathcal R(-x^\top\xi) \\
% 	\text{subject to } & x\in \mathcal S, \text{ and}\\
%                                 &  g_p(t, x^\top\xi) \le 0\\
%                               & \qquad  \text{for all $t$ satisfying } g_{p-1}(t, x^\top\xi)= 0.
% \end{align}
% The optimization problem~\eqref{eq:146} is non-linear with implicit inequality constraints.
% The objective is convex and, with Proposition~\ref{prop:167}, the constraint~\eqref{eq:49} is convex as well.
\subsection{Managing stochastic dominance constraints} 
% \todo{improve and add implicitly}

To solve the optimization problem in \eqref{eq:45}--\eqref{eq:32}, we apply Newton’s method in a context that includes uncountably many constraints.
A critical component of our approach involves effectively handling the {stochastic dominance constraints} with finitely many constraints (cf.\,\eqref{eq:146}).
However, implementing our restarted optimization introduces indirect challenges.

To understand the complexity of this problem, it is important to analyze the implicit relationship between \( t \) and \( x \) in~\eqref{eq:578}.
The variable \( t \) is defined implicitly by the equation \( g_{p-1}(t, x^\top \xi) = 0 \), which shows that \( t \) depends on \( x \) through a non-explicit functional relationship.
This relationship can be written as \( t(x) \).
%Incorporating this dependency into the inequality constraint \( g_p(t, x^\top \xi) \leq 0 \) adds a layer of difficulty.
Substituting \( t(x) \) into the constraint transforms it into \( g_p(t(x), x^\top \xi) \leq 0 \), where \( t(x) \) must satisfy \( g_{p-1}(t(x), x^\top \xi) = 0 \).

This implicit coupling between \( x \) and \( t \) introduces challenges for solving the optimization problem in \eqref{eq:146}--\eqref{eq:578}.
The feasibility of the constraints requires solving \( g_{p-1}(t, x^\top \xi) = 0 \) for \( t \), which embeds a root-finding process within the optimization framework.
% While the objective function is convex and the constraints maintain convexity, the implicit dependence of \( t(x) \) introduces complexity in both the analytical representation and the numerical computation of the solution.
Furthermore, substituting \( t(x) \) into the inequality constraint \( g_p(t(x), x^\top \xi) \leq 0 \) makes the constraint inherently non-linear.
This combination of implicit and non-linear features underscores the intricate nature of the problem and its deviation from standard formulations with explicit and linear constraints.
In the following discussion, we explain how we tackle this challenge.
\begin{figure}[h]
  \centering
  % First subplot
  \begin{subfigure}[t]{0.45\textwidth}
    \centering
    \begin{tikzpicture}[scale=1]
        
      % Draw axes
      \draw[thick,->] (-3.5,0) -- (3.5,0) node[right] {$t$};
      \draw[thick,->] (0,-3.5) -- (0,0.5) node[above] {$g_p(t, x^\top\xi)$};

      % Add grid (very thin and gray)
      %\draw[very thin, gray] (-3.5,-3.5) grid (3.5,0.5);
      
      % Define a strictly non-positive function g_p(t, X) with maximum at t=1
      \draw[domain=-1.5:3,smooth,variable=\x,blue,thick] 
          plot ({\x},{-0.5*(\x-1)^2 + 0}) 
          node[below left] {$g_p(t, x^\top\xi)$};

      % Special points on the t axis
      \fill[red] (-1,0) circle (2pt) node[below] {$t_1$};
      \fill[red] (1,0) circle (2pt) node[below] {$t_2$};  % This is where g_p(t2) = 0
      \fill[red] (2.5,0) circle (2pt) node[below] {$t_3$};

      % Corresponding special points on the function (g_p(t_2, X) = 0)
      \fill[red] (-1,{ -0.5*(-1-1)^2 + 0 }) circle (2pt) node[above left] {$g_p(t_1, x^\top\xi)$};
      \fill[red] (1,0) circle (3pt) node[above right] {$\overline{g}_p(x)$};  % Zero point at t = 1
      \fill[red] (2.5,{ -0.5*(2.5-1)^2 + 0 }) circle (2pt) node[below left] {$g_p(t_3, x^\top\xi)$};

      % Mark the zero point explicitly with a dashed line
      \draw[dashed] (1,0) -- (1,0);

    \end{tikzpicture}
    \caption{Function $\overline{g}(x)$ in an optimal configuration where $g_p(t, x^\top\xi) \leq 0 \quad \forall\,t \in \mathbb R$ \label{fig:FunG1}}
  \end{subfigure}
  \hfill
  %\hspace{0.2cm}
  % Second subplot
  \begin{subfigure}[t]{0.45\textwidth}
    \centering
    \begin{tikzpicture}[scale=0.75]
        
      % Draw axes
      \draw[thick,->] (-3.5,0) -- (4.5,0) node[right] {$t$};
      \draw[thick,->] (0,-3.5) -- (0,1.5) node[above] {$g_p(t, x^\top\xi)$};

      % Add grid (very thin and gray)
      %\draw[very thin, gray] (-3.5,-4.5) grid (4.5,1.5);
      
      % Shade the positive region where g_p(t, x) > 0
      \fill[fill=orange, opacity=0.3] (-1.5,0) -- plot[domain=0:2, smooth] ({\x},{-0.5*(\x-1)^2 + 0.5}) -- (2.5,0) -- cycle;

      % Define a function g_p(t, X) with part of it going into the positive side, indicating non-optimal x
      \draw[domain=-1.5:4,smooth,variable=\x,blue,thick] 
          plot ({\x},{-0.5*(\x-1)^2 + 0.5}) 
          node[below left] {$g_p(t, x^\top\xi)$};

      % Special points on the t axis
      \fill[red] (-1,0) circle (2pt) node[below] {$t_1$};
      \fill[black] (1,0) circle (2pt) node[below left] {$\hat{t}$};  
      \fill[red] (2,0) circle (2pt) node[below] {$t_2$};  
      \fill[red] (3.5,0) circle (2pt) node[below] {$t_3$};

      % Corresponding special points on the function
      \fill[red] (-1,{ -0.5*(-1-1)^2 + 0.5 }) circle (2pt) node[above left] {$g_p(t_1, x^\top\xi)$};
      \fill[red] (2,{ 0 }) circle (3pt) node[above right] {$\overline{g}_p(x) = 0$};  % Peak at t = 1
      \fill[red] (3.5,{ -0.5*(3.5-1)^2 + 0.5 }) circle (2pt) node[below left] {$g_p(t_3, x^\top\xi)$};

    \end{tikzpicture}
    \caption{Function $\overline{g}(x)$ in a configuration where $g_p(t, x^\top\xi) \nleq 0 \quad \forall\,t \in \mathbb R$ \label{fig:FunG2}}
   \end{subfigure}
  
  % Main caption
  \caption{Function configurations of $\overline{g}(x)$, cf.\ \eqref{eq:51}, illustrating various values of $t$, where $g_p(\tilde{t}, x^\top \xi) \leq 0$ and $g_{p-1}(\tilde{t}, x^\top \xi) = 0$ for $\tilde{t} \in \{t_1, t_2, t_3\}$, including cases where $g_p(t, x^\top\xi)$ satisfies or does not satisfy the condition $\forall\,t \in \mathbb{R}$.}
\end{figure}
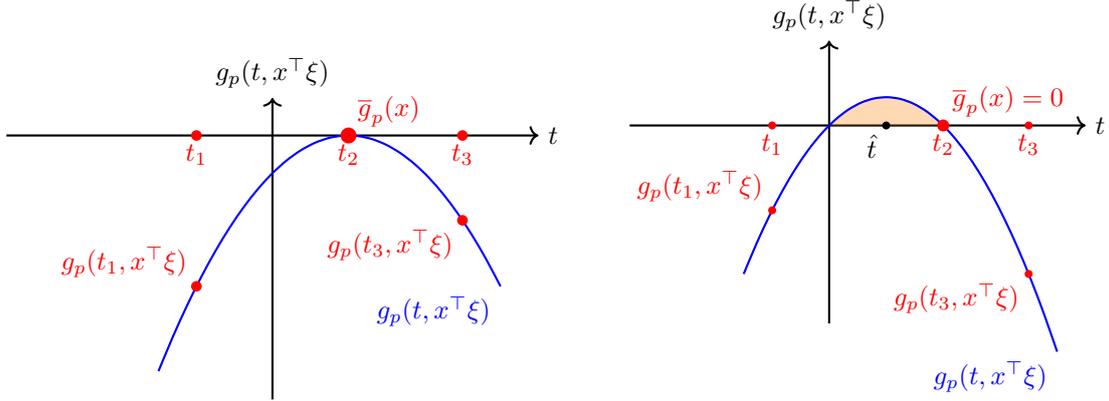

\newpage
We define the function 
\begin{align}\label{eq:53}
  g_p(x)&\coloneqq  \E(t-x^\top\xi)_+^p- \E(t-\xi_0)_+^p
  \shortintertext{and, accordingly,}
  \overline g_p(x)&\coloneqq  \max_{t\in \mathbb R} g_p(x)= \max_{t\in\mathbb R}\left\{ \E(t-x^\top\xi)_+^p- \E(t-\xi_0)_+^p \right\},
\end{align}
to frame the optimization problem as
\begin{align}\label{eq:50}
	\text{minimize } & \mathcal R(-x^\top\xi) \\
	\text{subject to } & x\in \mathcal S, \text{ and}\\
                                &  \overline g_p(x) \le 0.
\end{align}
With Proposition~\ref{prop:167}, the function $\overline g_p$ is
\begin{align}\label{eq:51}
  \overline g_p(x)&= \max \bigl\{g_p(t,x)\colon g_{p-1}(t,x)=0\bigr\},
\end{align}
where we have employed the auxiliary function
\begin{equation}\label{eq:54}  g_p(t,x)\coloneqq \E(t-x^\top\xi)_+^p-\E(t-\xi_0)_+^p.\end{equation}
We have reduced the problem of finitely many stochastic dominance constraints~\eqref{eq:578} to a single constraint focused on $\overline{g}_p(x)$.
The function $\overline{g}_p$ can be explicitly expressed as
\begin{equation}\label{eq:58}
 \overline{g}_p(x) = g_p\bigl(t(x), x\bigr),
\end{equation}
where \( t(x) \) is implicitly defined by
\begin{equation}\label{eq:56}
  g_{p-1}\bigl(t(x), x\bigr) = 0.
\end{equation}
Hence, the derivative is
\begin{equation}\label{eq:60}
\nabla \overline{g}(x) = \frac{\partial g_p}{\partial t}\bigl(t(x), x\bigr) \cdot \nabla t(x) + \frac{\partial g_p}{\partial x}\bigl(t(x), x\bigr).
\end{equation}
To find the gradient of \( t \), we differentiate~\eqref{eq:56}
\[
0 = \frac{\partial g_{p-1}}{\partial t}\bigl(t(x), x\bigr) \cdot \nabla t(x) + \nabla_x g_{p-1}\bigl(t(x), x\bigr),
\]
where \( \nabla_x g_p = \left(\frac{\partial g_p}{\partial x_1}, \dots, \frac{\partial g_p}{\partial x_d}\right) \).
Using the \emph{implicit function theorem}, we obtain
\[
\nabla t(x) = -\frac{\nabla g_{p-1}\bigl(t(x), x\bigr)}{\frac{\partial g_{p-1}}{\partial t}\bigl(t(x), x\bigr)},
\]
which gives the explicit form of~\eqref{eq:60} as
\begin{equation}\label{eq:57}
\nabla \overline{g}(x) = -\frac{\partial g_p}{\partial t}\bigl(t(x), x\bigr) \frac{\nabla_x g_{p-1}\bigl(t(x), x\bigr)}{\frac{\partial g_{p-1}}{\partial t}\bigl(t(x), x\bigr)} + \frac{\partial g_p}{\partial x}\bigl(t(x), x\bigr).
\end{equation}
For the partial derivatives of \( g(t, x) \) in~\eqref{eq:54} and~\eqref{eq:57}, we have
\begin{align}
  \nabla_x g_p\bigl(t,x\bigr)&= -p \E\bigl((t-x^\top\xi)_+^{p-1}\cdot\xi\bigr)
                         \shortintertext{ and}
  \frac{\partial g_p\bigl(t,x\bigr)}{\partial t}&= p \E(t-x^\top\xi)_+^{p-1}- p \E(t-\xi_0)_+^{p-1}= p\, g_{p-1}(t,x).
\end{align}

\paragraph{Key idea.}
To confirm that a candidate \( x \) is indeed optimal, we only need to show that \( g_p(t, x^\top \xi) \) stays non-positive for all \( t \in \mathbb{R} \) by checking a single crucial point, \( \overline{g}(x) \) (see Figure~\ref{fig:FunG1}). This result follows from Theorem~\ref{thm:156} (cf.\,\eqref{eq:51}).

\subsection{Derivations and algorithm}
This section outlines our optimization technique and the algorithm for solving the problem~\eqref{eq:50} using Newton's method.
The optimization problem involves both the objective risk function (cf.~\eqref{eq:4}) and several constraints, including a stochastic dominance constraint.

To solve the optimization problem using Newton's method, we construct the following Lagrangian function, incorporating both the objective function and the constraints from the problem~\eqref{eq:50}.
The Lagrangian is
\begin{align}\label{eq:47}
  L(x, q; \lambda, \mu, \nu_1, \dots, \nu_d) &\coloneqq q + \frac{1}{1 - \beta} \bigl(\mathbb{E} (-x^\top \xi - q)_+^p \bigr)^{\frac{1}{p}} \\
  &\qquad + \lambda \left(1 - \sum_{i=1}^d x_i \right) + \sum_{i=1}^d \mathbf{1}_{x_i \leq 0} \, \nu_i \cdot x_i \\
  &\qquad + \mu \cdot \overline{g}(x).
\end{align}
This formulation allows us to address the objective risk function (cf.\,\eqref{eq:4}), the simplex constraints on \( x \) and as well as the stochastic dominance constraint through the function \( \overline{g}(x^\top \xi) \).
The partial derivatives of the Lagrangian~\eqref{eq:47} are
\begin{align}
   \frac{\partial L}{\partial x}&= -\frac{1}{1-\beta}\bigl(\E(-x^\top\xi-q)_+^p\bigr)^{\frac1 p-1}\cdot  \E\left((-x^\top\xi-q)_+^{p-1}\cdot\xi\right)
     -\lambda\cdot\one + \nu\bullet \one_{x<0} \\
     & \quad + \mu\cdot\nabla \overline g(x), \label{eq:48}\\
  \frac{\partial L}{\partial q}&= 1- \frac1{1-\beta} \left(\E\bigl(-x^\top\xi-q\bigr)_+^p\right)^{\frac1 p-1} \cdot \E\bigl(-x^\top\xi-q\bigr)_+^{p-1}, \\
   \frac{\partial L}{\partial\lambda}&=  1-\sum_{i=1}^d x_i, \\
  \frac{\partial L}{\partial\mu}&=  \overline g(x),\label{eq:315}
                  \shortintertext{and}
  \frac{\partial L}{\partial\nu}&=  x\, \one_{x<0},\label{eq:55}
\end{align}
where $a\bullet b$ is the elementwise (Hadamard) product and
$\one_{x<0}= \left(\begin{cases}1 & \text{if }x_i<0\\ 0 & \text{else}\end{cases} \right)_{i=1,\dots,n}$.
It follows from $\frac{\partial L}{\partial\nu}=0$ in~\eqref{eq:55} that $x\ge0$ so that the term $\nu\bullet\one_{x\le0}=0$ for feasible $x$ in~\eqref{eq:48} with $x\ge0$, which represents complementary slackness.
The function~$F(x, q, \lambda, \mu, \nu) \coloneqq \left( \frac{\partial L}{\partial x}, \frac{\partial L}{\partial q}, \frac{\partial L}{\partial \lambda}, \frac{\partial L}{\partial \mu}, \frac{\partial L}{\partial \nu} \right)$ requires evaluating~$\overline g(x)$ in~\eqref{eq:315} and its derivative $\nabla\overline g(x)$ in~\eqref{eq:48}, which~\eqref{eq:57} provides.

We extend the derivatives by~\eqref{eq:56} and the function to be solved by the parameter~$t$.
With this, the first order conditions can be restated as
\begin{align}
   0&= -\frac{1}{1-\beta}\bigl(\E(-x^\top\xi-q)_+^p\bigr)^{\frac1 p-1}\cdot  \E\left((-x^\top\xi-q)_+^{p-1}\cdot\xi\right)
     -\lambda\cdot\one + \nu\bullet \one_{x<0}\\  & \quad+ \mu\cdot\overline g_x(x,t)\label{eq:59}\\
  0&= 1- \frac1{1-\beta} \left(\E\bigl(-x^\top\xi-q\bigr)_+^p\right)^{\frac1 p-1} \cdot \E\bigl(-x^\top\xi-q\bigr)_+^{p-1}\\
   0&=  1-\sum_{i=1}^d x_i, \\
  0&=  x\, \one_{x<0},
  \shortintertext{and}
  0&=  g_p(t,x),\text{ cf.\ \eqref{eq:315} and \eqref{eq:58}, and}\\
  0&= g_{p-1}(t,x),\text{ cf.\ \eqref{eq:56}},
\end{align}
where $\overline g_x(x,t)$ in~\eqref{eq:59} is given by
\[  \overline g_x(x,t)\coloneqq -\frac{\partial g_p}{\partial t}\bigl(t,x\bigr)\frac{\nabla_x g_{p-1}\bigl(t,x\bigr)}{\frac{\partial g_{p-1}}{\partial t}\bigl(t,x\bigr)} +\frac{\partial g_p}{\partial x}\bigl(t,x\bigr),\]
cf.~\eqref{eq:57}.
\begin{remark}[Boundaries of $t$]
  To confirm that a candidate \( x \) is truly optimal, we need to ensure that \( g_p(t, x^\top \xi) \leq 0 \) for all \( t \in \mathbb{R} \).
  However, based on results from \citet{dentcheva2003optimization}, it is sufficient to check this non-positivity condition over the interval \( t \in [\essinf \xi_0, \esssup \xi_0]  \), where $\xi_0$ is the benchmark variable.
  This interval effectively captures the critical values of \( t \) where violations may occur, thus eliminating the need to evaluate \( g_p(t, x^\top \xi) \) over all of \( \mathbb{R} \).
  Despite this reduction, the interval \( [\essinf \xi_0, \esssup \xi_0] \) still contains infinitely many points, which poses a challenge in practice.
\end{remark}

\paragraph{Verifications.}
In practice, there are cases where \( g_p(t, x^\top \xi) \) becomes positive in certain localized regions for specific values of \( t \) (see Figure~\ref{fig:FunG2}).
%This indicates that the candidate \( x \) does not fully satisfy the constraint over the entire domain.
%To address this issue, we introduce an additional constraint to the optimization problem. This new constraint ensures that \( g_p(t, x^\top \xi) \) remains non-positive not only at \( \overline{g}(x) \) but also within the critical neighborhood where positivity is detected.
To tackle this issue, in addition to solving the function \( F(x, q, \lambda, \mu, \nu) \) to find its zero, we introduce a constraint. 
This new constraint addresses cases where \( g_p(t, x^\top \xi) \) becomes positive in certain regions.
The constraint focuses on controlling the maximum value of \( g_p(t, x^\top \xi) \) over any interval where positivity occurs.
To formalize this, we define an indicator function \( \hat{g}_p(x, \hat{t}) \) as follows
\[
\hat{g}_p(x, \hat{t}) \coloneqq
\begin{cases}
    g_p(x, \hat{t}), & \text{if } g_p(x, \hat{t}) > 0, \\
    0, & \text{if } g_p(x, \hat{t}) \leq 0,
\end{cases}
\]
where \( \hat{t} = \arg\max_{t \in [p_a, p_b]} g_p(t, x^\top \xi) \) and \( [p_a, p_b] \subset [\essinf \xi_0, \esssup \xi_0] \) represents the subset of \( t \) values where \( g_p(t, x^\top \xi) \) is positive.

By enforcing the constraint \( \hat{g}_p(x, \hat{t}) = 0 \) through iterative refinement, we ensure that \( g_p(t, x^\top \xi) \) does not exceed zero in any region \( [\essinf \xi_0, \esssup \xi_0] \) where it would otherwise be positive. 
This ensures non-positivity across the critical interval \( [\essinf \xi_0, \esssup \xi_0] \) and compliance with stochastic dominance constraints.%, instead of depending solely on isolated evaluations such as \( \overline{g}(x) \).

% Through iterative refinement, this method steers \( x \) towards a configuration where \( g_p(t, x^\top \xi) \leq 0 \) for all \( t \in [\essinf \xi_0, \esssup \xi_0] \), thereby fulfilling the global stochastic dominance constraints and achieving true optimality. Algorithm~\ref{alg:Newton} encapsulates the individual steps again.

% By incorporating this adjustment, we refine the solution iteratively.
% The additional constraint helps eliminate any regions where \( g_p(t, x^\top \xi) \) may exceed zero and ensures that the solution satisfies the required global stochastic dominance constraints.

\begin{algorithm}[!thb]
	\scriptsize
  \KwIn{Initialize \( x^{(0)} \), \( t^{(0)} \) \( q^{(0)} \), and multipliers \(\lambda^{(0)}, \mu^{(0)}, \nu^{(0)} \), tolerance \(\epsilon\), and interval \( [\min \xi_0, \max \xi_0] \)}
	Set iteration counter \( k = 0 \)\\

		\While{\emph{not converged}}{
			Compute the Jacobian matrix \( J \) by taking the derivative of \( F(x, q, \lambda, \mu, \nu) \) with respect to \( (x, q, \lambda, \mu, \nu) \)
			\[
			J = \nabla F(x, q, \lambda, \mu, \nu)
			\]
      Use the implicit function \(\nabla t(x)\) (see \eqref{eq:57}) to calculate \(\overline{g}_x(x,t)\) %to compute \(F(x, q, \lambda, \mu, \nu)\)\\

			\If{\(\|F(x, q, \lambda, \mu, \nu)\| < \epsilon\)}{
				\textbf{break} the loop as convergence is achieved
			}
			
			Compute the Newton step $\Delta$ by solving the system \(J\cdot \Delta = -F(x, q, \lambda, \mu, \nu) \)\\
			Update variables:
			\[
			x^{(k+1)} = x^{(k)} + \Delta_x, \quad q^{(k+1)} = q^{(k)} + \Delta_q, \quad t^{(k+1)} = t^{(k)} + \Delta_t
			\]
			Update multipliers:
			\[
			\lambda^{(k+1)} = \lambda^{(k)} + \Delta_\lambda, \quad \mu^{(k+1)} = \mu^{(k)} + \Delta_\mu, \quad \nu^{(k+1)} = \nu^{(k)} + \Delta_\nu
			\]
			Increment \( k \leftarrow k + 1 \)

		% Check for positive region after convergence
		% Check for positive region in \( g_p(x^{(k)}, t) \) over \( t \in [\min \xi_0, \max \xi_0] \):\\
		% \If{\emph{positive region $[p_a,p_b]$ exists}}{
		% 	Identify the maximum point \( \hat{t} = \arg\max_{t \in [p_a,p_b]} g_p(x^{(k)}, t) \) where \( g_p(x^{(k)}, \hat{t}) > 0 \)\\
		% 	Update the constraint:
		% 	\[
		% 	\hat{g}_p(x^{(k)}, \hat{t})			\]
		% 	Reset \( k = 0 \) and rerun the algorithm with the new constraint\\
		% }
	}
	\caption{Newton's method for stochastic dominance optimization with Risk function (cf.\ref{def:141}) \label{alg:Newton}}
	\KwOut{Optimal values of \( x \), \( q \), and multipliers \( \lambda \), \( \mu \), \( \nu \)}
\end{algorithm}
\paragraph{Other Objectives.}
Note that Algorithm~\ref{alg:Newton} computes not only the optimal \( x \), but also the values of \( q \) and \( t \), along with the multipliers \( \lambda \), \( \mu \), and \( \nu \).
Our proposed algorithm achieves this in a single, unified optimization process. Additionally, Algorithm~\ref{alg:Newton} is flexible and can be adapted to other objective functions.

For instance, consider maximizing the expected return
\[\max_{x} \,\mathbb{E} \, x^{\top} \xi  \]
subject to higher order stochastic dominance constraints.
In this case, the Lagrangian is modified, and the partial derivatives become
\begin{align}
  \frac{\partial L}{\partial x} &= \mathbb{E}\,\xi  - \lambda \cdot\one + \nu \bullet\one_{x < 0} + \mu \cdot \nabla \overline{g}(x), \label{eq:828} \\
 \frac{\partial L}{\partial \lambda} &= 1 - \sum_{i=1}^d x_i, \\
 \frac{\partial L}{\partial \mu} &= \overline{g}(x), \shortintertext{and} 
 \frac{\partial L}{\partial \nu} &= x \, \one_{x < 0}.
\end{align}
This modified formulation allows us to solve for the optimal values while adhering to the same constraints, thereby extending the algorithm’s applicability to a broader range of objectives.

\section{Numerical exposition\label{sec:numerics}}
This section presents the numerical exposition of our method and Algorithm~\ref{alg:Newton}, which solves the higher order stochastic dominance problem.
We provide all datasets and the Julia code used in the experiments to ensure transparency and allow others to replicate our results.
The corresponding implementations are available in the following GitHub repository
\begin{center}
  \url{https://github.com/rajmadan96/HigerOrderStochasticDominance.git}\,.
  \end{center}
  
%\footnote{\url{https://github.com/rajmadan96/StochasticDominance.git}}
%Our tests use real-world financial data, demonstrating the practical application and efficiency of our method. This enables a comprehensive comparison with standard approaches and highlights the advantages of our algorithm.
%the Center for Research in Security Prices (CRSP), specifically focusing on the 49 industry portfolios provided by the Fama $\&$ French Data Library.\footnote{\url{http://mba.tuck.dartmouth.edu/pages/faculty/ken.french/data_library.html}}
%For our preliminary analysis, we select five sectors and limit the data to one month.

%\subsection{Small dataset, portfolio and benchmark\label{sec:dataBenchMark}}
\subsection{Higher order stochastic dominance}
This section exposes higher-order stochastic dominance using our proposed method.

\subsubsection{Mainstream big data}

\begin{figure}[!htb]
	\centering
	\includegraphics[width=0.6\textwidth]{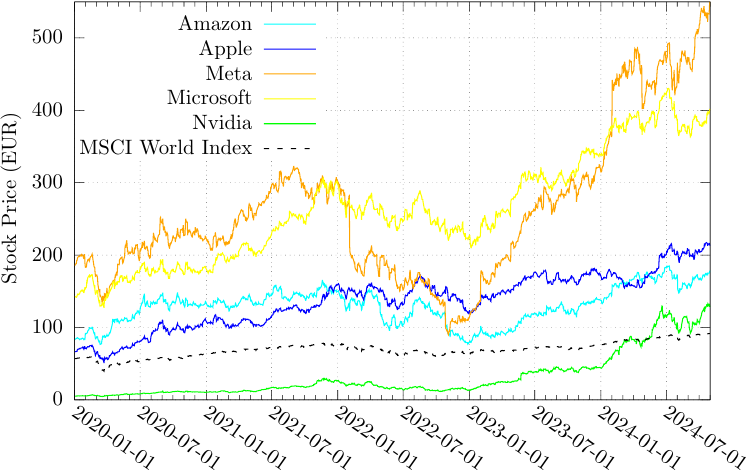}
	\caption{MSCI World index and 5 big tech stock prices: Jan 2020 – Oct 2024}\label{fig:BigTech}
\end{figure}

This section analyzes portfolio optimization using five leading tech companies.
Namely, Amazon, Apple, Meta (formerly Facebook), Microsoft, and Nvidia.
The chosen portfolio holds the largest allocation within the MSCI World Index.
The dataset covers the period from January 1, 2020, to October 31, 2024, with performance measured at two-month intervals.
This period includes the global COVID-19 pandemic and its aftermath, a time of significant market volatility and economic shifts, making it an insightful timeframe for analysis.
The analysis includes a total of 28 scenarios across 5 assets.
For the analysis, the MSCI World Index serves as the benchmark. We obtained price data using the \enquote*{YFinance} package in Julia and transformed these prices into returns to facilitate our study.
\begin{figure}[!h]
  \centering
  % Top figure: Maximized Portfolio Return
  \begin{minipage}{\textwidth}
  %  \centering
  \begin{tikzpicture}
    \begin{axis}[
        width=\textwidth,
        height=6cm,
        xlabel={stochastic order},
        ylabel={portfolio return ($\%$)},
        grid=major,
        xtick={1,2,3,4,5,6,7},
        xticklabels={$2^{\text{nd}}$, $3^{\text{rd}}$, $9^{\text{th}}$, $10^{\text{th}}$, $15^{\text{th}}$, $20^{\text{th}}$,  $\infty$},
        ymin=0.00, ymax=90.00, % Adjusting the y-axis range for return
        legend style={at={(1.05,0.5)}, anchor=west},
        scale=0.7
    ]
  
    % Maximized Portfolio Return (blue line)
    \addplot[color=blue, mark=o] coordinates {
        (1,  34.565934398103806) % 2nd
        (2,  37.11540649716656) % 6rd
        (3, 42.282366387901796) % 11th
        (4, 45.439942114052535) % 16th
        (5, 51.1867950685384) % 21th
        (6, 64.26113258958645) % 21th
        (7, 84.30496696443277) % infinity
    };
    \addlegendentry{Objective}
    %((1+ 1.661017097711228/100)^(12/2) -1)*100 convert to Annualized return 
    \addplot[color=red, mark=+] coordinates {
        (1, 10.389229614786167) % 2nd
        (2, 10.389229614786167) % 3rd
        (3, 10.389229614786167) % 5th
        (4, 10.389229614786167) % 10th
        (5, 10.389229614786167) % 15th
        (6, 10.389229614786167) % 15th
        (7, 10.389229614786167) % infinity
    };
    \addlegendentry{MSCI World}
    % \addlegendentry{MSCI returns}
    % \addplot[color=green, mark=*] coordinates {
    %   (1, 19.87) % 2nd
    %   (2, 19.87) % 3rd
    %   (3, 19.87) % 5th
    %   (4, 19.87) % 10th
    %   (5, 19.87) % 15th
    %   (6, 19.87) % 15th
    %   (7, 19.87) % infinity
    % };
      \end{axis}
  \end{tikzpicture}
  \end{minipage}
  
  % Bottom figure: Portfolio Allocations
  \begin{minipage}{\textwidth}
    %\centering
    \begin{tikzpicture}
      \begin{axis}[
          width=\textwidth,
          height=6cm,
          xlabel={stochastic order},
          ylabel={portfolio allocation (\%)},
          grid=major,
          xtick={1,2,3,4,5,6,7,8},
          xticklabels={$2^{\text{nd}}$, $3^{\text{rd}}$, $9^{\text{th}}$, $10^{\text{th}}$, $15^{\text{th}}$, $20^{\text{th}}$,  $\infty$},
          legend style={at={(1.05,0.5)}, anchor=west, legend columns=1},
          ymin=0, ymax=110,
          scale=0.7
      ]
    
      % Nvidia Allocation (Highest in some orders, scaled to 100%)
      \addplot[fill=orange, opacity=0.3,draw=orange] coordinates {
        (1, 100.0)
        (2, 100.0)
        (3, 100.0)
        (4, 100.0)
        (5, 100.0)
        (6, 100)
        (7, 100.0)    
      }\closedcycle;
      \addlegendentry{Nvidia}%fill=orange
      % Apple Allocation (Scaled relative to highest in each order)
      \addplot[fill=red, opacity=0.3, draw=red] coordinates {
        (1, 78.83774957080088)
        (2, 74.64433096201148)
        (3, 65.0)
        (4, 60.0)
        (5, 50.0)
        (6, 30.0)
        (7, 0.0012)  
        }\closedcycle;
      \addlegendentry{Apple}

      % Microsoft Allocation (Scaled relative to highest in each order)
      \addplot[fill=green, opacity=0.3,draw=green] coordinates {
          (1, 41.20433507883062)
          (2, 35.35717102128317)
          (3, 35.0)
          (4, 30.0)
          (5, 30.0)
          (6, 10.0)
          (7, 0.0)
      }\closedcycle;
      \addlegendentry{Microsoft}

      % Meta Allocation (Scaled relative to highest in each order)
      \addplot[fill=blue, opacity=0.3,draw=blue] coordinates {
        (1, 3.1020155763874406)
        (2, 4.8745280147823085)
        (3, 7.677214416294423)
        (4, 5.0)
        (5, 5.0)
        (6, 5.0)
        (7, 0.0)
      }\closedcycle;
      \addlegendentry{Meta}
    
      % Amazon Allocation (Lowest in each order, scaled)
      \addplot[fill=purple, opacity=0.3,draw=purple] coordinates {
        (1, 0)
        (2, 0)
        (3, 2.322785583705577)
        (4, 0)
        (5, 0)
        (6, 0.0)
        (7, 0.0)
      }\closedcycle;
      \addlegendentry{Amazon}
    \end{axis}
    \end{tikzpicture}
 \end{minipage}

  \caption{Maximized portfolio (annualized) return (top) and cumulative portfolio allocation by stochastic order (bottom)\label{fig:HigherOrderBigData}}
\end{figure}
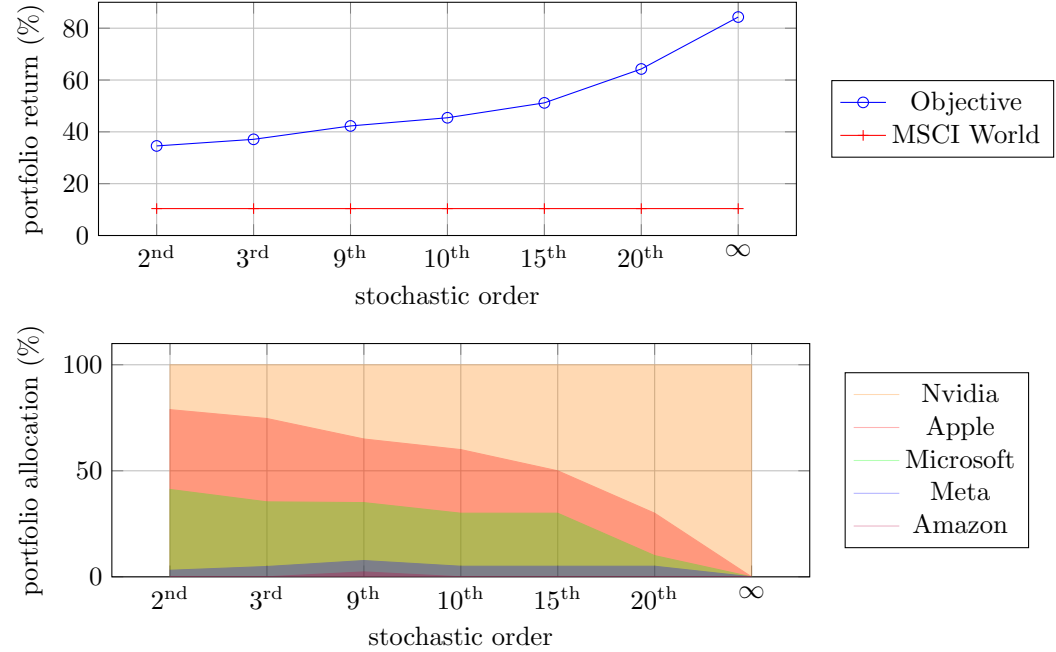

Figure~\ref{fig:HigherOrderBigData} presents outcomes of our experiment.\footnote{We convert bimonthly returns into annualized returns} 
In the top plot, the portfolio's annualized return (marked as “Objective”) consistently increases as we move to higher orders, reaching around 80$\%$ at the final point.
In the bottom plot, we observe the portfolio allocation for each asset at various orders.
Initially, Nvidia and Apple have smaller allocations in the portfolio, but their allocation grows significantly as the order increases.
By the highest orders, Nvidia and Apple dominate the portfolio, suggesting they play a larger role in driving returns at these levels.
In contrast, Microsoft, Meta, and Amazon start with higher allocations in the lower orders, but their contributions steadily decrease, becoming minimal by the highest orders.
This distribution indicates that as the stochastic order increases, the portfolio relies more heavily on Nvidia and Apple to achieve higher returns.

\subsubsection{Analysis using standardized dataset}
For this analysis, we use the dataset provided by the \enquote*{Fama $\&$ French Data Library}.\footnote{\url{http://mba.tuck.dartmouth.edu/pages/faculty/ken.french/data_library.html}}
See Appendix~\ref{sec:data} for more details. 
The portfolio returns, \(\xi \in \mathbb{R}^{22 \times 5}\), is a matrix, where the 22 rows represent the scenarios and the 5 columns represent the assets.
We construct the benchmark returns, \(\xi_0\), by averaging the returns across the five sectors
\[
\xi_0 = \sum_{j=1}^{5} \xi_{ij} \tau_j \quad \text{for }  i=1,\dots,22,
\]
where each sector weight, \(\tau_j = \frac{1}{5}\), is equally distributed.
%employed in the performance comparison (see Appendix~\ref{sec:data}).
For this experiment, we again focus on simple objective function, \( \max_{x} \frac{1}{n}\sum_{i=1}^{n} \sum_{j=1}^{d} \xi_{ij} x_j \).
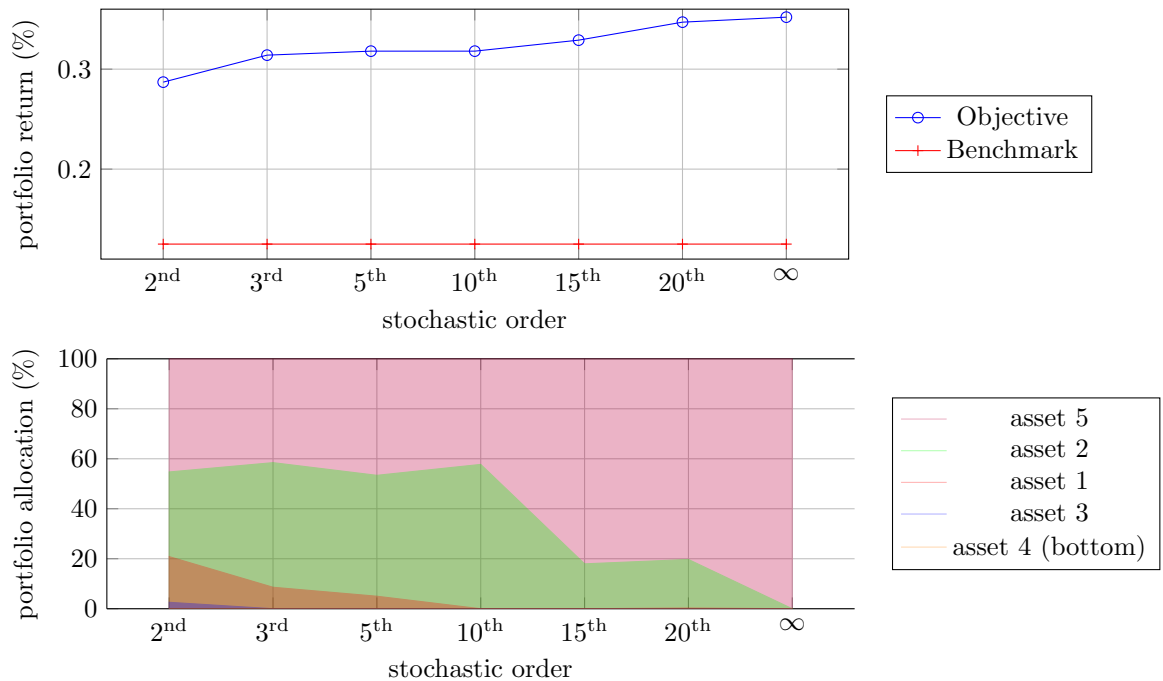
\begin{figure}[ht]
  \centering
  
  % Top figure: Maximized Portfolio Return
  \begin{minipage}{\textwidth}
  %  \centering
    \begin{tikzpicture}
      \begin{axis}[
          width=\textwidth,
          height=6cm,
          xlabel={stochastic order},
          ylabel={portfolio return ($\%$)},
          grid=major,
          xtick={1,2,3,4,5,6,7},
          xticklabels={$2^{\text{nd}}$, $3^{\text{rd}}$, $5^{\text{th}}$, $10^{\text{th}}$, $15^{\text{th}}$, $20^{\text{th}}$,$\infty$},
          ymin=0.11, ymax=0.36, % Adjusting the y-axis range for return
          legend style={at={(1.05,0.5)}, anchor=west},
          scale=0.75
      ]
    
      % Maximized Portfolio Return (blue line)
      \addplot[color=blue, mark=o] coordinates {
          (1, 0.287) % 2nd
          (2, 0.314) % 3rd
          (3, 0.318) % 5th
          (4, 0.318) % 10th
          (5, 0.329) % 15th
          (6, 0.347) % 20th
          (7, 0.352) % infinity
      };
      \addlegendentry{Objective}
      \addplot[color=red, mark=+] coordinates {
          (1, 0.125) % 2nd
          (2, 0.125) % 3rd
          (3, 0.125) % 5th
          (4, 0.125) % 10th
          (5, 0.125) % 15th
          (6, 0.125) % 20th
          (7, 0.125) % infinity
      };
      \addlegendentry{Benchmark}
    
      \end{axis}
    \end{tikzpicture}
  \end{minipage}
  
  %\vspace{0.2cm} % Space between the plots

  % Bottom figure: Portfolio Allocations
  \begin{minipage}{\textwidth}
    \centering
    \begin{tikzpicture}
      \begin{axis}[
          width=\textwidth,
          height=6cm,
          xlabel={stochastic order},
          ylabel={portfolio allocation (\%)},
          grid=major,
          xtick={1,2,3,4,5,6,7},
          xticklabels={$2^{\text{nd}}$, $3^{\text{rd}}$, $5^{\text{th}}$, $10^{\text{th}}$, $15^{\text{th}}$, $20^{\text{th}}$,$\infty$},
          legend style={at={(1.05,0.5)}, anchor=west, legend columns=1}, % Legend on the right side
          scale=0.75,
          ymin=0, ymax=100, % Adjusted range for percentage
          axis y line*=left,
      ]
    
      % Asset 5 Allocation (cumulative of all others)
      \addplot[fill=purple, opacity=0.3, draw=purple] coordinates {
          (1, 100)
          (2, 100)
          (3, 100)
          (4, 100)
          (5, 100)
          (6, 100)
          (7, 100)
      } \closedcycle;
      \addlegendentry{{asset 5}}
    
      % Asset 2 Allocation (cumulative excluding Asset 5)
      \addplot[fill=green, opacity=0.3, draw=green] coordinates {
          (1, 54.73)
          (2, 58.46)
          (3, 53.39)
          (4, 57.75)
          (5, 17.97)
          (6, 19.72)
          (7, 0)
      } \closedcycle;
      \addlegendentry{{asset 2}}
    
      % Asset 1 Allocation (cumulative excluding Asset 5 and Asset 2)
      \addplot[fill=red, opacity=0.3, draw=red] coordinates {
          (1, 20.86)
          (2, 8.60)
          (3, 5.01)
          (4, 0.00)
          (5, 0.00)
          (6, 0.30)
          (7, 0)
      } \closedcycle;
      \addlegendentry{{asset 1}}
    
      % Asset 3 Allocation (cumulative excluding Asset 5, Asset 2, and Asset 1)
      \addplot[fill=blue, opacity=0.3, draw=blue] coordinates {
          (1, 2.58)
          (2, 0.00)
          (3, 0.00)
          (4, 0.00)
          (5, 0.00)
          (6, 0.00)
          (7, 0)
      } \closedcycle;
      \addlegendentry{{asset 3}}
    
      % Asset 4 Allocation (at the bottom)
      \addplot[fill=orange, opacity=0.3, draw=orange] coordinates {
          (1, 0)
          (2, 0)
          (3, 0)
          (4, 0)
          (5, 0)
          (6, 0)
          (7, 0)
      } \closedcycle;
      \addlegendentry{{asset 4 (bottom)}}
    
      \end{axis}
    \end{tikzpicture}
  \end{minipage}

  \caption{Maximized portfolio return (top) and cumulative portfolio allocation by order (bottom)\label{fig:HigherOrder}}
\end{figure}
Figure~\ref{fig:HigherOrder} illustrates how maximized portfolio returns and asset allocations evolve across different stochastic dominance orders, from the  $2^{\text{nd}}$ order up to infinity.
The maximized portfolio return by order graph shows a clear upward trend.
The portfolio return starts at around 0.30$\%$ for the $2^{\text{nd}}$ order and gradually increases to approximately 0.35$\%$ as the stochastic order  reaches infinity.
This trend suggests that higher-order stochastic dominance strategies lead to better portfolio returns.
The portfolio allocation by order provides a breakdown of how allocations to different sectors change with increasing order.
We observe that
\begin{itemize}[nolistsep]
\item asset 1 and asset 3 maintain small and relatively stable allocations across orders.
\item The allocation of asset 2 increase at first but then gradually declines as the order increases.
\item Asset 5 starts with a high allocation and becomes dominant as the order approaches infinity, receiving nearly 100$\%$ of the allocation.
\item Asset 4 consistently receives no allocation across all orders.
\end{itemize}
This combination of charts shows that as the dominance order increases, the portfolio becomes more concentrated in \emph{asset 5}, leading to higher returns at higher orders.
This suggests a shift from diversification toward focusing on the highest-performing asset as the order increases.

\paragraph{Risk function (cf.\,Definition \ref{def:141}).}
Next, we proceed with the same experimental setup but apply a different objective, the risk function (see, Definition~\ref{def:141}), specifically focusing on the case \( p = 1 \) (second-order cases).
Table~\ref{tab:Risksecond} presents the second-order stochastic dominance results for different values of \( \beta \). 

As \( \beta \) increases, we see notable shifts in the allocations.
For asset 1 and asset 5, higher~\( \beta \) values result in reduced allocations, while asset 2 sees a substantial increase at \( \beta = 0.8 \).
This change suggests that as we raise \( \beta \), the focus shifts toward sectors like asset 2 and asset 4 to satisfy risk preferences. 
%At \( \beta = 0.1 \), the **Asset 3** sector has the highest allocation (0.3852), showing preference for it when risk aversion is low. However, as \( \beta \) rises to 0.5 and 0.8, **Asset 3**'s allocation drops, reflecting a shift toward higher-risk sectors, particularly **Asset 5**, which reaches its peak at \( \beta = 0.5 \).
The risk function \( \mathcal{R}_{\beta} \) also grows significantly as \( \beta \) increases, indicating a higher tolerance for risk across these scenarios.

% These results suggest that increasing \( \beta \) leads to a more diversified allocation, favoring sectors with moderate to high risk as we move away from low-risk options.
\begin{table}[!htb]
  \footnotesize
  \centering
  \begin{tabular}{cccccccc}
  \hline
  $\beta$ & \textbf{asset 1} & \textbf{asset 2} & \textbf{asset 3} & \textbf{asset 4} & \textbf{asset 5}& Risk Function $\mathcal R_{\beta} $ cf.\,\eqref{def:141} \\ \hline
  0.1   & 0.2108 &  0.1855 & 0.3852 & 0.0589 & 0.1491 & 0.0742\\
  0.5   & 0.1920 & 0.1224 & 0.2128 & 0.0834 & 0.3639 &0.6645 \\
  0.8 & 0.1442 & 0.3965 & 0.1758 & 0.1955 & 0.0788 & 1.1211 \\
  \hline
  \end{tabular}
  \caption{Second order stochastic dominance for different $\beta$ values \label{tab:Risksecond}}
\end{table}

\subsection{Performance comparison}
Prominent numerical algorithms are available for second and third order stochastic dominance in the literature.
We compare the numerical solutions of these well-established approaches and the number of stochastic dominance constraints required with our proposed methods.
This comparison study is conducted to demonstrate the consistency of our approach.

\subsubsection{Second-order stochastic dominance}
First, we focus on the second-order stochastic dominance methods introduced by  \citet{dentcheva2003optimization} and \citet{kuosmanen2004efficient}.
%For this experiment, we restrict the analysis to the case where \( p = 1 \) 

When \( p = 1 \), the functions \( g_{\xi_0}(t) = \mathbb{E}(t - \xi_0)_+ \) and \( g_{x^{\top}\xi}(t) = \mathbb{E}(t - x^{\top}\xi)_+ \) may have discontinuous derivatives, which complicates the standard proof as explained in Remark~\ref{rem:167}.
However, verifying the inequality 
\[
\mathbb{E}(t - \xi_0)_+ \geq \mathbb{E}(t - x^{\top}\xi)_+ \quad \text{for all } t \in \mathbb{R},
\]
only requires checking specific values of \( t \).
These values occur where 
\[
\mathbb{P}(\xi_0 \leq t) = \mathbb{P}(x^{\top}\xi \leq t),
\]
where \( \mathbb{P}(\xi_0 \leq t) = \mathbb{E}[\mathbf{1}_{\{t > \xi_0\}}] \) and \( \mathbb{P}(x^{\top}\xi \leq t) = \mathbb{E}[\mathbf{1}_{\{t > x^{\top}\xi\}}] \). Theorem~\ref{thm:156} confirms this result for the case \( p = 1 \), as the presence of well-defined jumps in the functions compensates for the lack of continuity in their derivatives, as detailed in Remark~\ref{rem:225}. 
%In this experiment, we use the benchmark, portfolio, objective, and dataset as outlined in the previous section.

We base our analysis on a dataset from \citet[Section~8]{dentcheva2003optimization}, which includes eight assets spanning a 22-year period. 
A detailed dataset and descriptive statistics can be found in Appendix~\ref{sec:data}.
The portfolio return \(\xi \in \mathbb{R}^{22 \times 8}\) is a matrix, where the 22 rows represent the years (scenarios) and the 8 columns represent the assets.
We construct the benchmark returns,~\(\xi_0\), by averaging the returns across the eight assets:
\[
\xi_0 = \sum_{j=1}^{8} \xi_{ij} \tau_j \quad \text{for }  i=1,\dots,22,
\]
where each asset weight, \(\tau_j = \frac{1}{8}\), is equally distributed.
%We selected the small dataset deliberately to focus on a detailed analysis of the numerical intricacies.
Additionally, we use a simple objective function \begin{equation}\label{eq:695} \max_{x} \frac{1}{n}\sum_{i=1}^{n} \sum_{j=1}^{d} \xi_{ij} x_j. \end{equation} 
% \( \xi_0 = \sum_{j=1}^{d} \xi_{ij} \tau_j \quad \text{for } i = 1, \dots, n \), where \( \tau_j = \frac{1}{d} \) represents equally probable weights.

%We simplify our formulation by focusing on the same objective function (cf.\,\eqref{eq:694}), replacing the risk measure \(\mathcal{R}_\beta(-x^\top\xi)\).
\begin{table}[!h]
  \scriptsize
  \centering
  \begin{tabular}{lccc}
    \toprule
    \textbf{asset} & \textbf{Our approach}& \citet{dentcheva2003optimization} & \citet{kuosmanen2004efficient}\\ \midrule
    1 & 0.0& 0.0& 0.0 \\
    2 & 0.0 & 0.0 & 0.0\\
    3 & 0.0680& 0.0680 & 0.0680 \\
    4 & 0.1880 & 0.1880 & 0.1880\\
    5 & 0.0& 0.0& 0.0\\
    6 & 0.3913 & 0.3913 & 0.3913 \\
    7 &  0.2309 & 0.2309 & 0.2309\\
    8 & 0.1216 & 0.1216 & 0.1216\\ \midrule
    \multicolumn{1}{c}\textbf{objective}&11.00 $\%$&11.00$\%$ & 11.00$\%$ \\ \midrule
    \multicolumn{1}{c}\textbf{number of SSD constraints}&\textbf{3}& 506 & 550 \\ \bottomrule
 \end{tabular}
 \caption{Second order stochastic dominance (case $p=1$): a comparison with prominent approaches\label{tab:CompareProminent}}
\end{table}
Table~\ref{tab:CompareProminent} compares our approach with the algorithms of \citet{dentcheva2003optimization} and \citet{kuosmanen2004efficient} in terms of optimal portfolio allocations, objective values, and the number of SSD constraints.
Our proposed method achieves the same portfolio weights and reaches the optimal objective value effectively.
However, in this case,  our approach only needs~\emph{3~SSD constraints}. 
In general, \citet{dentcheva2003optimization} algorithm requires \( n^2 + n \) SSD constraints, and \citet{kuosmanen2004efficient} algorithm requires \( n^2 + 3n \) SSD constraints.
This means that their methods become much more complex as the number of scenarios (\( n \)) grows.
In contrast, our approach keeps the number of constraints significantly low, making it faster and more suitable for large-scale problems.
This significant reduction in constraints shows that our method can solve similar problems with less computational effort.
%\footnote{In general, \citet{dentcheva2003optimization} algorithm has $n^2+n$ SD constraints}
%\footnote{In general, \citet{kuosmanen2004efficient} algorithm has $n^2+3n$ SD constraints}

% Our formulation is designed for \(p > 1\), allowing us to apply it starting from interger \(p = 2\), which corresponds to third-order stochastic dominance. In this experiment, we compare our method with the Kuosmanen test, which is based on second-order stochastic dominance. The goal is to show that our approach not only works for higher-order dominance but also aligns closely with well-established approach.

% From the table, we see that the optimal allocations for most sectors are quite similar between the two approaches. For example, the Asset 2 sector shows almost identical results, while the Asset 1ulture and Asset 5 sectors display slight differences. However, these differences are expected since we are working with a higher-order dominance.
% %Notably, our approach slightly outperforms in terms of maximized portfolio return, achieving 0.3142 compared to 0.3085 under the Kuosmanen test.
% The result substantiate that our method is consistent with Kuosmanen's SSD test, which is a prominent existing method.
\begin{figure}[ht]
  \centering
  \includegraphics[width=0.5\textwidth]{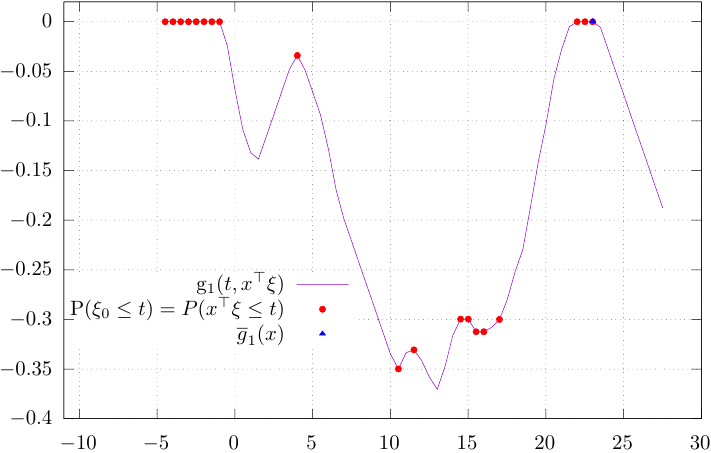}
  \caption{Stochastic dominance function \( g_1(t, x^\top \xi) \) for case $p=1$}
  \label{fig:p=1}
  \vspace{1em} % Adds some space between the figure and the table
\end{figure}
%\todo{highlight the max point and update in text}

Figure~\ref{fig:p=1} shows that the stochastic dominance function \( g_1(t, x^\top \xi) \) (purple line, see equation~\eqref{eq:531}) remains in the non-positive region, confirming second-order stochastic dominance.
The red dots represent points where \(\mathbb{P}(\xi_0 \leq t) = \mathbb{P}(x^\top \xi \leq t)\), indicating alignment in cumulative probabilities.

According to our theoretical framework, verifying stochastic dominance requires only checking points that satisfy equation~\eqref{eq:37}.
Our algorithm goes a step further by calculating the specific point \( \overline{g}_1(x) \) (blue triangle), which conclusively confirms the dominance relation.

\subsubsection{Third-order stochastic dominance}
We next consider the third-order stochastic dominance (TSD) method by \citet{ThirdKopa}.
This method uses an advanced portfolio optimization framework to approximate strict TSD conditions.
It breaks the portfolio return space into threshold levels and expresses semivariance and expected shortfall as piecewise-quadratic functions.
\begin{table}[!h]
  \scriptsize
  \centering
  \begin{tabular}{lccc}
    \toprule
    \textbf{asset} & \textbf{Our approach}& \citet{ThirdKopa}\\ \midrule
    1 & 0.0& 0.0\\
    2 & 0.0 & 0.0 \\
    3 & 0.2549& 0.2621 \\
    4 & 0.0022 & 0.0 \\
    5 & 0.0& 0.0\\
    6 & 0.3763 & 0.4081 \\
    7 &  0.2485 & 0.2403 \\
    8 &0.1180 & 0.0892 \\ \midrule
    \multicolumn{1}{c}\textbf{objective}&11.03 $\%$&11.02$\%$ \\ \midrule
    \multicolumn{1}{c}\textbf{number of TSD constraints}&\textbf{3}& 508\\ \bottomrule
 \end{tabular}
 \caption{Third order stochastic dominance (case $p=2$): a comparison with \citet{ThirdKopa}\label{tab:CompareProminent2}}
\end{table}
We use the same experimental setup as in the previous section.
Table~\ref{tab:CompareProminent2} compares our TSD method with the approach by \citet{ThirdKopa}.
%Our method employs only three constraints, whereas theirs uses 508, demonstrating its significant computational efficiency.
The results indicate that our solution aligns closely with the solutions of \citet{ThirdKopa}.
Both approaches assign almost similar weights to the critical assets, such as 3, 6, 7, and 8, reflecting a consistent understanding of the portfolio optimization problem under TSD conditions.
For example, our method assigns a weight of 0.3763 to asset 6, compared to 0.4081 by \citet{ThirdKopa}.
However, our approach achieves this alignment with significantly fewer constraints—just three, compared to 508 in \citet{ThirdKopa}.
This highlights the computational efficiency of our method while maintaining comparable precision in capturing the essential characteristics of the TSD conditions.
% We use the same experimental setup as in the previous section. Table~\ref{tab:CompareProminent2} compares our TSD method with the approach by \citet{ThirdKopa}. Our method uses only three constraints, while theirs uses 508, showing its efficiency.
% Our method assigns zero weights to non-contributing assets 1, 2, 3, and 5. In contrast, \citet{ThirdKopa} assigns a small weight to asset 1. For important assets like 4 and 6, our method assigns higher weights, showing it better reflects their value under exact TSD conditions.
% This comparison shows that our method offers improved precision and computational efficiency while addressing some limitations of approximation-based methods.
% \todo{Add result and interpretation of Risk measures; Add remark that other methods are computational expesive}
\section{Conclusion}
In summary, our work provides a new characterization of stochastic dominance, presenting two formulations that enable handling stochastic dominance constraints of higher orders, including non-integral order.
Further, the results simplify the verification by reducing the required test points to a finite set.

We also introduce a practical optimization framework that incorporates non-linear stochastic dominance constraints.
Finally, we validate our approach with a numerical study that demonstrates the reliability of Algorithm~\ref{alg:Newton} for solving higher-order stochastic dominance problems.

Future research could focus on extending the algorithm to encompass a wider range of stochastic dominance constraints and exploring its application in more advanced scenarios.
For example, one could consider settings involving multiple objective or multiple dominance constraints (cf.\ \citet{gutjahr2016stochastic}, \citet[Section~6]{dentcheva2003optimization}).
Future work may include a theoretical discussion on determining the exact number of constraints required to ensure consistency.

\section{Acknowledgement}
The authors thank the Euro Summer Institute 2024 in Ischia for its valuable insights and discussions that contributed to this research.

\bibliographystyle{abbrvnat}
\bibliography{LiteraturAlois.bib,LiteraturRaj.bib}
%\newpage
\appendix
\newpage
\section{Dataset and descriptive statistics \label{sec:data}}
%\subsection{Dataset 1 \label{sec:Sdata1}}
%The following table presents the yearly returns for the eight selected industry portfolios 
\begin{table}[!htbp]
  \tiny
  \centering
  \caption{Returns and descriptive statistics for selected assets over 22 years (cf.\ \citet{dentcheva2003optimization})}
  \begin{tabular}{ccccccccc}
  \toprule
  \textbf{Year/Metric} & \textbf{asset 1} & \textbf{asset 2} & \textbf{asset 3} & \textbf{asset 4} & \textbf{asset 5} & \textbf{asset 6} & \textbf{asset 7} & \textbf{asset 8} \\
  \midrule
  \multicolumn{9}{l}{\textit{Annual Returns}} \\
  1  & $7.5$  & $-5.8$ & $-14.8$ & $-18.5$ & $-30.2$ & $2.3$  & $-14.9$ & $67.7$ \\
  1  & $8.4$  & $2$    & $-26.5$ & $-28.4$ & $-33.8$ & $0.2$  & $-23.2$ & $72.2$ \\
  3  & $6.1$  & $5.6$  & $37.1$  & $38.5$  & $31.8$  & $12.3$ & $35.4$  & $-24$  \\
  4  & $5.2$  & $17.5$ & $23.6$  & $26.6$  & $28$    & $15.6$ & $2.5$   & $-4$   \\
  5  & $5.5$  & $0.2$  & $-7.4$  & $-2.6$  & $9.3$   & $3$    & $18.1$  & $20$   \\
  6  & $7.7$  & $-1.8$ & $6.4$   & $9.3$   & $14.6$  & $1.2$  & $32.6$  & $29.5$ \\
  7  & $10.9$ & $-2.2$ & $18.4$  & $25.6$  & $30.7$  & $2.3$  & $4.8$   & $21.2$ \\
  8  & $12.7$ & $-5.3$ & $32.3$  & $33.7$  & $36.7$  & $3.1$  & $22.6$  & $29.6$ \\
  9  & $15.6$ & $0.3$  & $-5.1$  & $-3.7$  & $-1$    & $7.3$  & $-2.3$  & $-31.2$\\
  10 & $11.7$ & $46.5$ & $21.5$  & $18.7$  & $21.3$  & $31.1$ & $-1.9$  & $8.4$  \\
  11 & $9.2$  & $-1.5$ & $22.4$  & $23.5$  & $21.7$  & $8$    & $23.7$  & $-12.8$\\
  12 & $10.3$ & $15.9$ & $6.1$   & $3$     & $-9.7$  & $15$   & $7.4$   & $-17.5$\\
  13 & $8$    & $36.6$ & $31.6$  & $32.6$  & $33.3$  & $21.3$ & $56.2$  & $0.6$  \\
  14 & $6.3$  & $30.9$ & $18.6$  & $16.1$  & $8.6$   & $15.6$ & $69.4$  & $21.6$ \\
  15 & $6.1$  & $-7.5$ & $5.2$   & $2.3$   & $-4.1$  & $2.3$  & $24.6$  & $24.4$ \\
  16 & $7.1$  & $8.6$  & $16.5$  & $17.9$  & $16.5$  & $7.6$  & $28.3$  & $-13.9$\\
  17 & $8.7$  & $21.2$ & $31.6$  & $29.2$  & $20.4$  & $14.2$ & $10.5$  & $-2.3$ \\
  18 & $8$    & $5.4$  & $-3.2$  & $-6.2$  & $-17$   & $8.3$  & $-23.4$ & $-7.8$ \\
  19 & $5.7$  & $19.3$ & $30.4$  & $34.2$  & $59.4$  & $16.1$ & $12.1$  & $-4.2$ \\
  20 & $3.6$  & $7.9$  & $7.6$   & $9$     & $17.4$  & $7.6$  & $-12.2$ & $-7.4$ \\
  21 & $3.1$  & $21.7$ & $10$    & $11.3$  & $16.2$  & $11$   & $32.6$  & $14.6$ \\
  22 & $4.5$  & $-11.1$& $1.2$   & $-0.1$  & $-3.2$  & $-3.5$ & $7.8$   & $-1$   \\
  \midrule
  \multicolumn{9}{l}{\textit{Descriptive Statistics}} \\
  Mean          & $7.81$   & $9.29$  & $11.97$  & $12.36$ & $12.13$ & $9.17$ & $14.12$ &  $8.35$  \\
  Median        & $7.6$   & $5.5$   & $13.25$ & $13.7$   & $16.35$ & $7.8$ &  $11.3$ &  $-0.2$   \\
  Std. Dev.     & $3.04$   & $15.21$   & $16.81$   & $17.86$   & $22.36$ & $8.05$ & $23.54$ &  $26.34$   \\
  Variance      & $9.27$   & $231.52$   & $282.82$  & $319.13$   & $500.06$ & $64.89$ & $554.22$ & $693.92$  \\
  Skewness      & $0.73$  & $0.85$   & $-0.45$  & $-0.48$  & $-0.29$ &  $0.82$ &  $0.41$ &  $0.93$ \\
  Kurtosis      & $0.27$   & $-0.05$  & $-0.48$   & $-0.45$   & $-0.07$ & $0.62$ & $-0.03$ &  $0.54$   \\
  Minimum       & $3.1$  & $-11.1$  & $-26.5$  & $-28.4$  & $-33.8$ & $-3.5$ & $-23.4$ & $-31.2$ \\
  Maximum       & $15.6$    & $46.5$    & $37.1$    & $38.5$    & $59.4$ & $31.1$ &  $69.4$ &    $72.2$   \\
  Percentile 25 & $5.8$ & $-1.72$    &  $2.2$  & $0.5$  & $-2.65$  &    $2.47$ & $-0.8$ &  $-7.7$  \\
  Percentile 75 & $9.07$   & $18.85$     & $23.3$   & $26.35$      &  $26.42$ &  $14.8$ & $27.37$ & $21.5$   \\
  \bottomrule
  \end{tabular}
  \end{table}
  %\subsection{Dataset 2\label{sec:Sdata2}}
%The following tables present the daily returns for the five selected industry portfolios and the detailed descriptive statistical analysis. 

\begin{table}[!htbp]
  \tiny
  \centering
  \caption{Daily returns and descriptive statistics for selected portfolios}
  \begin{tabular}{ccccccc}
  \toprule
  \textbf{Date/Metric} & \textbf{asset 1} & \textbf{asset 2} & \textbf{asset 3} & \textbf{asset 4} & \textbf{asset 5} \\
  \midrule
  \multicolumn{6}{l}{\textit{Daily Returns}} \\
  2024-07-01 & $-1.01$ & $-0.72$ & $1.10$ & $-1.80$ & $-0.65$ \\
  2024-07-02 & $-2.50$ & $-0.22$ & $-0.86$ & $-0.09$ & $0.64$ \\
  2024-07-03 & $-0.38$ & $0.27$ & $-0.21$ & $0.41$ & $0.41$ \\
  2024-07-05 & $-1.11$ & $0.15$ & $-0.33$ & $2.10$ & $-1.61$ \\
  2024-07-08 & $0.44$ & $-0.22$ & $-0.64$ & $0.59$ & $0.33$ \\
  2024-07-09 & $0.05$ & $-1.49$ & $1.52$ & $-2.59$ & $-0.26$ \\
  2024-07-10 & $-0.34$ & $0.79$ & $1.31$ & $2.75$ & $1.93$ \\
  2024-07-11 & $4.00$ & $1.60$ & $2.53$ & $0.73$ & $3.72$ \\
  2024-07-12 & $1.76$ & $1.04$ & $-0.06$ & $0.41$ & $2.17$ \\
  2024-07-15 & $1.53$ & $0.02$ & $-3.42$ & $-2.50$ & $-0.63$ \\
  2024-07-16 & $2.77$ & $1.28$ & $2.59$ & $0.40$ & $1.83$ \\
  2024-07-17 & $0.71$ & $0.69$ & $-1.44$ & $1.12$ & $0.94$ \\
  2024-07-18 & $-2.24$ & $-1.24$ & $-1.74$ & $-1.72$ & $-1.78$ \\
  2024-07-19 & $0.08$ & $-1.51$ & $-0.23$ & $-2.67$ & $-1.81$ \\
  2024-07-22 & $0.35$ & $0.29$ & $-0.54$ & $-0.42$ & $-0.25$ \\
  2024-07-23 & $2.55$ & $-0.09$ & $0.08$ & $-0.39$ & $1.14$ \\
  2024-07-24 & $-2.21$ & $2.88$ & $-1.61$ & $-2.51$ & $-0.41$ \\
  2024-07-25 & $1.67$ & $-0.21$ & $1.05$ & $0.00$ & $2.21$ \\
  2024-07-26 & $0.17$ & $1.26$ & $2.29$ & $0.93$ & $2.13$ \\
  2024-07-29 & $-0.97$ & $-0.79$ & $-0.48$ & $-0.48$ & $-1.13$ \\
  2024-07-30 & $-0.11$ & $0.79$ & $0.98$ & $-1.65$ & $0.07$ \\
  2024-07-31 & $0.24$ & $1.89$ & $0.72$ & $-1.14$ & $-1.23$ \\
  \midrule
  \multicolumn{6}{l}{\textit{Descriptive Statistics}} \\
  Mean          & $0.25$   & $0.29$   & $0.12$   & $-0.39$  & $0.35$   \\
  Median        & $0.12$   & $0.21$   & $-0.13$  & $-0.24$  & $0.20$   \\
  Std Dev       & $1.66$   & $1.11$   & $1.50$   & $1.53$   & $1.51$   \\
  Variance      & $2.74$   & $1.24$   & $2.25$   & $2.35$   & $2.29$   \\
  Skewness      & $0.34$   & $0.29$   & $-0.22$  & $0.09$  & $0.37$  \\
  Kurtosis      & $-0.22$   & $-0.23$   & $-0.16$   & $-0.75$   & $-0.68$   \\
  Minimum       & $-2.50$  & $-1.51$  & $-3.42$  & $-2.67$  & $-1.81$  \\
  Maximum       & $4.00$   & $2.88$   & $2.59$   & $2.75$   & $3.72$   \\
  Percentile 25 & $-0.82$  & $-0.22$  & $-0.61$  & $-1.70$  & $-0.64$  \\
  Percentile 75 & $1.32$   & $0.97$   & $1.08$   & $0.54$   & $1.65$   \\
  \bottomrule
  \end{tabular}
\end{table}

\end{document}